\documentclass[reqno,a4paper]{amsart}

\usepackage{graphicx}
\usepackage{color} 
\usepackage[usenames,dvipsnames]{xcolor}
\usepackage{transparent}
\usepackage[latin1]{inputenc}  
\usepackage{dsfont}
\usepackage{gensymb}
\usepackage[active]{srcltx}\usepackage[colorlinks,pdfpagelabels,pdfstartview=FitH,bookmarksopen=true,bookmarksnumbered=true,linkcolor=blue,plainpages=false,hypertexnames=false,citecolor=red]{hyperref}
\usepackage{mathrsfs}
\usepackage{verbatim, amssymb,amsmath, amsfonts, amsthm}
\usepackage{latexsym}
\usepackage[dvips]{epsfig}
\usepackage{float}

\allowdisplaybreaks 

\newtheorem{theorem}{Theorem}[section]
\newtheorem{proposition}[theorem]{Proposition}
\newtheorem{lemma}[theorem]{Lemma}
\newtheorem{corollary}[theorem]{Corollary}
\newtheorem{remark}[theorem]{Remark}



\DeclareMathOperator{\supp}{supp}

\newcommand{\N}{\mathbb{N}}

\newcommand{\R}{\mathbb{R}}
\renewcommand{\to}{\rightarrow}

\newcommand{\virg}[1]{\textquotedblleft#1\textquotedblright}

\newcommand{\apice}[1]{^{\raisebox{-0.5ex}{\fontsize{2}{2}\selectfont{#1}}}}

\begin{document}
\numberwithin{equation}{section}
\parindent=0pt
\hfuzz=2pt
\frenchspacing

\title[Morse index formula]{A Morse index formula\\ for  radial solutions of \\ Lane-Emden problems
}

\author[]{Francesca De Marchis, Isabella Ianni, Filomena Pacella}

\address{Francesca De Marchis, University of Roma {\em Sapienza}, P.le Aldo Moro 5, 00185 Roma, Italy}
\address{Isabella Ianni, Second University of Napoli, V.le Lincoln 5, 81100 Caserta, Italy}
\address{Filomena Pacella, University of Roma {\em Sapienza}, P.le Aldo Moro 5, 00185 Roma, Italy}



\thanks{Research partially supported by:  PRIN $201274$FYK7$\_005$ grant, INDAM - GNAMPA and Sapienza Funds \virg{Avvio alla ricerca 2015}}

\begin{abstract} We consider the semilinear Lane-Emden problem:
\begin{equation}\label{problemAbstract}\left\{\begin{array}{lr}-\Delta u= |u|^{p-1}u\qquad  \mbox{ in }B\\
u=0\qquad\qquad\qquad\mbox{ on }\partial B
\end{array}\right.\tag{$\mathcal E_p$}
\end{equation}
where $B$ is the unit ball of $\R^N$, $N\geq3$, centered at the origin and $1<p<p_S$, $p_S=\frac{N+2}{N-2}$. \\
We prove that for any  radial solution $u_p$ of \eqref{problemAbstract} with $m$ nodal domains  its Morse index $\mathsf{m}(u_p)$ is given by the formula
\[\mathsf{m}(u_p)=m+N(m-1)\]
if $p$ is sufficiently close to $p_S$.
\end{abstract}

\maketitle

\section{Introduction}\label{Introduction}
We consider the classical Lane-Emden problem
\begin{equation}\label{problem}\left\{\begin{array}{lr}-\Delta u= |u|^{p-1}u\qquad  \mbox{ in }B\\
u=0\qquad\qquad\qquad\mbox{ on }\partial B\subset \R^N
\end{array}\right.
\end{equation}
where $B$ is the unit ball of $\R^N$, $N\geq 3$, centered at the origin and $1<p<p_S$, with $p_S=\frac{N+2}{N-2}=2^*-1$, where $2^*$ is the critical exponent for the Sobolev embedding $H^1_0(B)\hookrightarrow L^{2^*}(B)$.

\

In this paper we study the Morse index of the  radial  solutions of \eqref{problem}.

\

We recall that the \emph{Morse index} $\mathsf{m}(u_p)$ of a solution $u_p$ of \eqref{problem} is the maximal dimension of a subspace $X\subset H_0^1(B)$ where the quadratic form associated to the linearized operator at $u_p$:
\[
L_{p}=(-\Delta-p|u_p|^{p-1})
\]
is negative definite. Equivalently, since $B$ is a bounded domain, $\mathsf{m}(u_p)$ can be defined as the number of the negative Dirichlet  eigenvalues of $L_{p}$ counted with their multiplicity.

\

It is well known that \eqref{problem} possess infinitely many radial solutions among which only one is positive (or negative) while all the others change sign and can be characterized by the number of their nodal regions. For a given radial solution $u_p$ of \eqref{problem} with $m$ nodal domains, it has been proved in \cite{HarrabiRebhibSelmi} that the \emph{radial Morse index}, i.e. the number of the negative eigenvalues of $L_{p}$ in the Sobolev space of radial functions  $H_{0,rad}^1(B)$, is exactly $m$. Obviously the Morse index $\mathsf{m}(u_p)$, in $H^1_0(B)$, can be larger than $m$, because of the presence of negative non radial eigenvalues of $L_{p}$.
\\
The knowledge of the Morse index is, in general, a very important qualitative property of a solution. In particular it helps to classify the solutions and study their stability or possible bifurcations.

\

A first estimate that we get for a radial solution $u_p$ of \eqref{problem} with $m$ nodal domains is the following one (see Theorem \ref{teo:AftPacGen}):
\begin{equation}\label{stimaIntroBasso}
\mathsf{m}(u_p)\geq m+N(m-1),
\end{equation}
which improves a result in \cite{AftalionPacella}.

\

The main theorem of the present paper states that for $p$ close to the critical exponent the extimate \eqref{stimaIntroBasso} is sharp. More precisely we prove:

\begin{theorem}\label{teoPrincipaleMorse}
Let $N\geq3$ and $u_p$ be a radial  solution to \eqref{problem} with $m\in\N^+$ nodal regions. Then
\begin{equation}\label{formulina}
\mathsf{m}(u_p)= m+ N(m-1), \qquad\mbox{ for $p$ sufficiently close to $p_S$.
}
\end{equation}
\end{theorem}

\

Let us make a few comments about this result pointing out some interesting features of the formula \eqref{formulina}.

\


First, writing \eqref{formulina}
as
\[\mathsf{m}(u_p)= m(N+1)-N,\]
we see that the Morse index $\mathsf{m}(u_p)$ grows linearly with respect to the number $m$ of nodal domains, which corresponds also to the number of negative radial  eigenvalues of the operator $L_{p}$ (\emph{cf.} \cite{HarrabiRebhibSelmi}).
This is somehow surprising since, in general, one would expect \emph{many more} negative nonradial eigenvalues then the negative radial ones. Indeed if we look at the distribution of the radial and nonradial eigenvalues of the linear operator $(-\Delta )$ in $H^1_0(B)$ we observe that:
\begin{itemize}
\item[(i)] on one side by a result of Br\"uning-Heintze and Donnelly \cite{Bruning1, Bruning2, Donnelly} we get that
\[\lambda_{r,m}\sim C m^2\quad \mbox{ as }m\rightarrow +\infty\]
where $\lambda_{r,m}$ is the $m$-th radial eigenvalues of $(-\Delta)$, which  implies that the number  $n_r(m^2)$ of the radial eigenvalues of $(-\Delta)$  bounded by $m^2$ is  $m$, more precisely
\[n_r(m^2)\sim  m \quad \mbox{ as $m\rightarrow +\infty$}\]
\item[(ii)] on the other side by the classical Weil law (see e.g. \cite{StraussBook}):
\[n(m^2)\sim C m^N\quad \mbox{ as $m\rightarrow +\infty\ $ ($N$ is the dimension)}\]
where $n(m^2)$ is the number of all the eigenvalues of $(-\Delta)$ in $H^1_0(B)$ less than or equal to $m^2$
\end{itemize}
In an equivalent way we can observe that if we consider a radial eigenfunction of $(-\Delta)$ in $H^1_0(B)$ with $m$ nodal regions, i.e. corresponding to the eigenvalue $\lambda_{r,m}$, then its \emph{Morse index} is just the number of the eigenvalues less than $\lambda_{r,m}$ which, by (i) and (ii),
 grows at a rate  of order $m^N$ and so faster then $m$ (if $N\geq 2$) as $m\rightarrow +\infty$.\\
 So $L_{p}$ represents an example of a linear, Schr\"odinger type, operator determined by the potential $V_p(x)=p|u_p(x)|^{p-1}$, for $p$ approaching $p_S$, for which (i) and (ii) do not hold, at least for the negative eigenvalues.

 \

Another interesting consequence of all this could be derived studying \eqref{problem} as $p\rightarrow 1$. In this case it is reasonable to conjecture the convergence of the Morse index $\mathsf{m}(u_p)$  to the Morse index of the Dirichlet radial eigenfunction of $(-\Delta)$ with $m$ nodal regions (i.e. the eigenfunction corresponding to the radial eigenvalue $\lambda_{r,m}$) possibly augmented by the multiplicity of $\lambda_{r,m}$, which is $1$.
Indeed suitable normalizations of solutions of \eqref{problem} converge to eigenfunctions of the Laplacian as $p\rightarrow 1$ (see \cite{Bonheure, {Grossi}}).  
Therefore the previous considerations indicate that  for large $m$ the Morse index $\mathsf{m}(u_p)$ for $p$ close to $1$ is of order $m^N$, hence it is much bigger than $m+ N(m-1)$, which is by \eqref{formulina} the Morse index of $u_p$ for $p$ close to $p_S$. So bifurcations from $u_p$ should appear, as $p$ ranges from $1$ to $p_S$, showing that the structure of the solution set of \eqref{problem} is richer than one could imagine.

%

\

Next we would like to point out another interesting fact: the formula \eqref{formulina} \emph{does not hold in dimension $N=2$, as $p\rightarrow p_S=+\infty$}. Indeed  in the recent paper \cite{DeMarchisIanniPacellaMathAnn} we have proved the following:

\begin{theorem}[\cite{DeMarchisIanniPacellaMathAnn}]\label{teoPrincipaleMorseN2}
Let $u_p$ be a radial sign-changing solution to \eqref{problem} with $2$ nodal regions, but with $B\subset\R^2$ and $p_S=+\infty$. Then
\[
\mathsf{m}(u_p)= 12 \qquad\mbox{ for $p$ sufficiently large. }
\]
\end{theorem}

\

Obviously $12\neq m+N(m-1)=4$ for $N=2$ and $m=2$. Note that in this case the value of $m(u_p)$ seems to be related to the Morse index of one of the radial solutions to the singular Liouville problem in $\R^2$ (\cite{ChenLin}), see \cite{DeMarchisIanniPacellaMathAnn} for further details.

\

Let us describe the method for proving Theorem \ref{teoPrincipaleMorse}, which also clarifies the differences with the case $N=2$.

\

Since the solutions $u_p$ are radial, to study the spectrum of the linearized operator $L_p$ we decompose it as a sum of the spectrum of a radial weighted operator and the spectrum of the Laplace-Beltrami operator on the unit sphere. To bypass the difficulty of dealing with a weighted eigenvalue problem with a singularity at the origin we approximate the ball $B$ by annuli $A_n$ with a small hole, showing that the number of negative eigenvalues of the linearized operator $L_p$ is preserved (we refer to \cite{DeMarchisIanniPacellaMathAnn} for this). Then (see Section \ref{Section:Approximation}) it turns out that the Morse index $\mathsf{m}(u_p)$ is determined by the \emph{size} of the first $(m-1)$ (radial) eigenvalues $\tilde{\beta}_i(p)$, $i=1,\ldots, m-1$, of the weighted operator
\begin{equation}\label{opw}
\tilde{L}^n_p=|x|^2(-\Delta-V_p(x))
\end{equation}
in $H_0^1(A_n)$, where the potential $V_p(x)$ is $p|u_p(x)|^{p-1}$ and $n=n_p$ is properly chosen.\\In order to study  these eigenvalues a good knowledge of the potential $V_p(x)$ is needed which, in turns, means to have accurate estimates on the solutions $u_p$. This is where the hypothesis on the exponent $p$ enters.\\
If $N\geq3$, in Section \ref{Section:Asymptotic2} we make a precise analysis of the asymptotic behavior of $u_p$ as $p\to p_S$, which allows to get the needed estimates on the potential $V_p(x)$ for $p$ close to the critical exponent.
In particular  we get that suitable rescalings of $u_p$ in each nodal region converge to the same positive radial solution $U$ of the critical equation in $\R^N$:
\begin{equation}\label{introeqcritica}
-\Delta U=U^{p_S}\ \mbox{ in }\mathbb R^N,\ N\geq3.
\end{equation}
This allows to detect precisely the asymptotic behavior, as $p\to p_S$, of the first eigenvalue $\tilde \beta_1(p)$ (and then, as a consequence, of all the other eigenvalues $\tilde \beta_i(p)$, $i=2,\ldots, m-1$) by several nontrivial estimates  (see Section \ref{Section:Asymptotic4}).

\

In dimension $2$ the procedure followed in \cite{DeMarchisIanniPacellaMathAnn} is similar but the striking difference with respect to the case $N\geq3$ is that the limit problems, as $p\to+\infty$, for the positive and negative part of the nodal radial solutions $u_p$ with $2$ nodal domains are different. Indeed it was proved in \cite{GrossiGrumiauPacella2} that (assuming w.l.g. $u_p(0)>0$) a suitable rescaling of $u_p^+$ converges to a regular solution of the Liouville problem in $\R^2$, while a suitable rescaling of $u^-_p$ converges to a radial solution of a singular Liouville problem in $\R^2$ (see also \cite{DeMarchisIanniPacellaJEMS}). So the estimates needed to compute the Morse index of $u_p$ are completely different and the contribution from the annular nodal region is bigger and makes the Morse index of $u_p$ higher with respect to the corresponding case in dimension $N\geq3$.  This difference reflects in the study of the asymptotic behavior of the first radial eigenvalue $\tilde{\beta}_1(p)$ (see Remark \ref{rmkDifferenceN2}) which makes the proof in dimension $N\geq 3$ more delicate than that for $N=2$. 
\\
We also point out that the assertion of Theorem \ref{teoPrincipaleMorse} holds for radial solutions to \eqref{problem} with \emph{any number of nodal regions}, while  in the case $N=2$ the result of \cite{DeMarchisIanniPacellaMathAnn} has been obtained only for solutions with $2$ nodal regions. This is because an asymptotic analysis of radial solutions with $m\geq 3$ is lacking in dimension $N=2$. We believe that the strategy of the present paper could be pursued also in dimension $N=2$ to get a result for general radial solutions. We plan to do this in a future paper.
%

\

A final comment is that the  whole strategy for the Morse index computation (here as in \cite{DeMarchisIanniPacellaMathAnn}) relies on the peculiar behavior of the radial solutions which have all the nodal regions shrinking at the same point as $p\rightarrow p_S$ (as $ p\rightarrow +\infty$ when $N=2$). This property also induces an interesting blow-up (in time) phenomenon in the associated parabolic problem with initial data close to the radial stationary solutions (see \cite{DeMarchisIanni, DicksteinPacellaSciunzi, MarinoPacellaSciunzi}).

\

The paper is organized as follows. We start in Section \ref{sectionLowerBounMorse} by proving a lower bound for the Morse index of radial solutions of semilinear elliptic Dirichlet problems with general autonomous nonlinearities. This part holds in any dimension $N\geq 2$ and extends previous results in \cite{AftalionPacella} giving, as a special case, the estimate \eqref{stimaIntroBasso}. In Section \ref{Section:Asymptotic2} we perform the asymptotic analysis of the radial  solutions of \eqref{problem} as $p\to p_S$. The results in this section are interesting in themselves and do not appear in previous papers. In Section \ref{Section:Approximation} we approximate the eigenvalue problem in the ball by corresponding ones in approximating annuli and set the auxiliary weighted eigenvalue problems. In section \ref{Section:Asymptotic4} we study the radial eigenvalues of the weighted operator $\tilde{L}^n_p$ introduced in  \eqref{opw}; in particular the analysis of the first one $\tilde{\beta}_1(p)$ is the central part of the section. The delicate estimates that we develop here are crucial for our proof; in order to obtain them we need to 
 analyze  accurately the contribution to the Morse index of each nodal region of $u_p$.
%
%
Finally  the proof of Theorem \ref{teoPrincipaleMorse} is presented in Section \ref{section:ProofMain}.

\tableofcontents

\

\section{A lower bound for the Morse index}\label{sectionLowerBounMorse}
We consider a semilinear elliptic problem with a general autonomous nonlinearity:
\begin{equation}
\label{problemaGenerale}
\left\{
\begin{array}{lr}
-\Delta u =f(u) & \mbox{ in }\Omega\\
u=0 & \mbox{ on }\partial\Omega
\end{array}
\right..
\end{equation} where $\Omega\subset\R^N$, $N\geq 2$ is  either a ball or an annulus centered at the origin and $f\in C^1(\R)$.

\

For a solution $u$ of \eqref{problemaGenerale} we denote by $\mathsf{m}(u)$ the \emph{Morse index} of $u$, namely the number of the negative Dirichlet eigenvalues of $L_u$ in $\Omega$ (counted with their multiplicity), where
$L_u: H^2(\Omega)\cap H^1_0(\Omega)\rightarrow L^2(\Omega)$ is the linearized operator at $u$, namely
\[ L_u( v): =   -\Delta v-f'(u(x))v.
\]
When the solution $u$ is radial we also denote by  $\mathsf{m_{rad}}(u)$ the \emph{radial Morse index} of $u$, i.e.  the number of negative radial eigenvalues of the linearized operator $L_{u}$.

\

We prove here a result which improves the one in \cite{AftalionPacella} and holds in any dimension $N\geq 2$.

\begin{theorem}\label{teo:AftPacGen} Let $u$ be a radial solution  of \eqref{problemaGenerale} with $m\geq 2$ nodal domains.
Then
\begin{equation}\label{tesi1}
\mathsf{m}(u)\geq \mathsf{m_{rad}}(u)+ N(m-1).
\end{equation}
Moreover, if $f$  satisfies  the condition \[f(s)\leq f'(s) s,\]
then  
\begin{equation}\label{tesi2}
\mathsf{m_{rad}}(u)\geq m
\end{equation}
and hence
\[\mathsf{m}(u)\geq m+ N(m-1).\]
\end{theorem}

\begin{proof}
Let us fix $m\in \N^+$ and let  us denote by $u_m$  a radial solution of \eqref{problemaGenerale} having $m$ nodal regions.  We use the partial derivatives of $u_m$ to produce negative eigenvalues whose corresponding eigenfunctions are odd with respect to an hyperplane passing through the origin. Let us consider, for any $i=1,\ldots,N$, the hyperplane $T_i=\{x=(x_1,\ldots,x_N)\,:\, x_i=0\}$ and the domain $\Omega_i^-=\{x\in \Omega\,:\, x_i<0\}$, i.e. $\Omega_i^-$ is the \emph{half ball} or the \emph{half annulus} determinated by $T_i$.\\
Then we denote by $A_1,\ldots,A_m$ the nodal regions of $u_m$, counting them starting from the outer boundary in such a way that $\partial A_1$ contains $\partial \Omega$ if $\Omega$ is a ball or  the outer boundary of $\Omega$ if $\Omega$ is an annulus. Since $u_m$ is radial we have that $A_j$ are annuli for $j\in\{1,\ldots,m-1\}$ while $A_m$ is a ball if $\Omega$ is a ball or another annulus if so is $\Omega$. Let us first consider the case of the ball so that:
\[
A_j=\{x\in \Omega\,:\, R_{j+1}<|x|<R_j\}\qquad j=1,\ldots,m-1
\]
\[
A_m=\{x\in \Omega\,:\, |x|<R_m\}
\]
where $R_j$, $j=2,\ldots,m$, are the nodal radii and $R_1$ is the radius of the ball $\Omega$.\\
We consider the derivatives $\frac{\partial u_m}{\partial x_i}$, $i=1,\ldots,N$, which satisfy the equation 
\begin{equation}
\label{eqderivparz}
L_{u_m}\left(\frac{\partial u_m}{\partial x_i} \right)=
0\qquad \mbox{in  $\Omega$.}
\end{equation}
Using the symmetry of $u_m$ we have:
\begin{equation}
\label{derivparzsuT_i}
\frac{\partial u_m}{\partial x_i}=0\qquad \mbox{on  $\overline \Omega\cap T_i$.}
\end{equation}
Then we consider the \emph{half nodal regions}
\[
A^-_{i,j}=A_j\cap \Omega^-_i,\qquad \mbox{$j=1,\ldots,m$ \ and \ $i=1,\ldots,N.$}
\]
To simplify the notations let us fix $i=1$ and focus on the function $\frac{\partial u_m}{\partial x_1}$ in the sets $A^-_{1,j}$, that we simply denote by $A_j^-$. Whatever we prove for $\frac{\partial u_m}{\partial x_1}$ will hold with obvious changes for the other derivatives $\frac{\partial u_m}{\partial x_i}$, $i=2,\ldots,N$.\\
Let us observe that for each nodal region $A_j$, writing $u_m(r)=u_m(|x|)$ there exists at least one value $r_j\in(R_j,R_{j+1})$, $j=1,\ldots,m-1$, such that
\begin{equation}
\label{derivparzsur_j}
\frac{d u_m}{d r}(r_j)=0.
\end{equation}
Notice that if the nonlinearity $f=f(s)$ satisfies the condition $s\,f(s)\geq0$ then $r_j$ is the unique radius in $(R_j,R_{j+1})$ such that \eqref{derivparzsur_j} holds in $A_j$, $j=1,\ldots,m-1$.\\
Then, since $u_m$ is radial we have that $\frac{\partial u_m}{\partial x_1}\equiv0$ on the spheres
\begin{equation}
\label{S_j}
S_j=\{x\in\R^N\,;\,|x|= r_j\}\qquad j=1,\ldots,m-1.
\end{equation}
Let us fix one $r_j\in(R_j,R_{j+1})$ for each $j=1,\ldots,m-1$ (i.e. just one value of the radius in the interval $(R_j,R_{j+1})$ such that \eqref{derivparzsur_j} holds) and consider the sets
\[
N_j^-=\{x\in\R^N\,:\, r_j>|x|>r_{j+1}\}\cap \Omega_1^-\qquad j=1,\ldots,m-2
\]
and observe that for $j=1,\ldots,m-2$, by \eqref{eqderivparz}, \eqref{derivparzsuT_i} and \eqref{S_j}
\begin{equation}
\label{eqderivparzsuN_j}
\left\{
\begin{array}{ll}
L_{u_m}\left(\frac{\partial u_m}{\partial x_1}\right)=0 &\mbox{in $N_j^-$}\\
\frac{\partial u_m}{\partial x_1}=0&\mbox{on $\partial N_j^-$.}
\end{array}
\right.
\end{equation}
Thus $\frac{\partial u_m}{\partial x_1}$ is an eigenfunction of the linearized operator $L_{u_m}$ in $N^-_j$ corresponding to the zero eigenvalue which is the first one or an higher one according to the fact that $\frac{\partial u_m}{\partial x_1}$ changes sign or not in $N^-_j$.\\
Moreover also in the set
\[
N^-_{m-1}=\{x\in\R^N\,:\, r_{m-1}>|x|\geq 0\}\cap \Omega^-_1
\]
the function $\frac{\partial u_m}{\partial x_1}$ satisfies \eqref{eqderivparzsuN_j} (for $j=m-1$). Hence also in $N^-_{m-1}$ zero is an eigenvalue for $L_{u_m}$ with corresponding eigenfunction $\frac{\partial u_m}{\partial x_1}$.\\
In conclusion we have obtained $(m-1)$ adjacent regions where an eigenvalue of $L_{u_m}$ is zero. This implies that in the domain $N^-=\cup_{j=1}^{m-1}N^-_j$ the $h$-th eigenvalue $\lambda_h$ of $L_{u_m}$ is zero for some $h\geq m-1$.\\
Since $N^-$ is strictly contained in $\Omega^-_1$, by construction we have that the $h$-th eigenvalue $\lambda_h$ of $L_{u_m}$ in $\Omega^-_{1}$ is negative for some $h\geq m-1$, in particular $\lambda_{m-1}=\lambda_{m-1}(L_{u_m})<0$ in $\Omega^-_1$ and so are all $\lambda_n=\lambda_n(L_m)$ in $\Omega^-_1$ for $n\leq m-1$. By reflecting by oddness with respect to $T_1$ the corresponding eigenfunctions we get eigenfunctions of $L_{u_m}$ in the whole $\Omega$ corresponding to the same $(m-1)$ negative eigenvalues $\lambda_n$, $n=1,\ldots,m-1$.

Repeating the same arguments for all $i=1,\ldots,N$ we get at least $(m-1)$ negative eigenvalues $\lambda_n(u_m)$ in the domains $\Omega^-_i$, for each $i=1,\ldots,N$, which give eigenvalues of $L_{u_m}$ in the whole $\Omega$ whose corresponding eigenfunctions are odd with respect to $T_i$, $i=1,\ldots,N$.\\ Note that, by symmetry,
\[
\lambda_n(L_{u_m}, \Omega^-_i)=\lambda_n(L_{u_m}, \Omega^-_s)\qquad\mbox{for $i\neq s$, \ $i,s=1,\ldots,N$ \quad $n=1,\ldots,m-1$}
\]
but the corresponding eigenfunctions are linearly independent, because they are odd with respect to orthogonal axes.\\
So the multiplicity of each eigenvalue $\lambda_n$ of $L_{u_m}$ in $\Omega$ is at least $N$ so that we have got at least $N(m-1)$ negative eigenvalues. Since the eigenfunctions we have found are not radial, adding $\mathsf{m_{rad}}(u_m)$, we get the estimate \eqref{tesi1}.

\

If f satisfies the condition $f(u)\leq f'(u)u$ then it is easy to see that each (radial) nodal region gives the existence of one negative radial eigenvalue, so we get \eqref{tesi2}.

\

The case when $\Omega$ is an annulus follows in a similar, slightly easier, way, since the only difference is that the last nodal region $A_m$ is an annulus, so that it does not need to be treated in a different way with respect to the other regions $A_j$, $j=1,\ldots, m-1$.
\end{proof}

\

We end this section recalling the following known result concerning the case when $f$ is a power type nonlinearity and the domain $\Omega$ is a ball (see \cite{BartschWeth} for the case $m=2$ and  \cite[Proposition 2.9]{HarrabiRebhibSelmi} for any $m\in\N^+$) 
\begin{theorem}[\cite{BartschWeth, HarrabiRebhibSelmi}]\label{teoMorseRadiale} 
Let $\Omega$ be a ball and $f(u)=|u|^{p-1}u$, $p\in (1,p_S)$, $p_S=\frac{N+2}{N-2}$ if $N\geq 3$, $p_S=+\infty$ if $N=2$.
Let $u$ be a radial solution to \eqref{problemaGenerale} with $m\in\N^+$ nodal regions. Then
\[\mathsf{m_{rad}}(u)=m.\]
\end{theorem}

\

\

\section{Asymptotic analysis of the nodal radial solutions}\label{Section:Asymptotic2}

In this section we analyze the asymptotic behavior as $p\rightarrow p_S$ of any radial sign-changing solution of \eqref{problem}. It is well known that for any fixed $p\in (1,p_S)$ the radial solutions of problem  \eqref{problem} are infinitely many, precisely for each $m\in\N^+$ there is a unique (up to the sign, being the nonlinearity odd) radial solution to \eqref{problem} with $m$ nodal domains.

\

So for $m\in\mathbb N^+$ let us denote by $u\apice{m}_p$ the unique nodal radial solution of \eqref{problem} having $m$ nodal regions and satisfying
\begin{equation}
\label{MaxInZero}
u\apice{m}_p(0)>0.
\end{equation}
The $1$-dimensional profile of this solution is described in Figure \ref{Figure}.
With abuse of notation we will  write often $u\apice{m}_p(r)=u\apice{m}_p(|x|)$.

\ 

In the next proposition we state a few qualitative properties of the solutions $u\apice{m}_p$.
\begin{proposition}\label{PropositionUnicoMaxeMin}
Let $p\in (1, p_S)$, then:
\begin{itemize}
\item[(i)] $u\apice{m}_p(0)=\|u\apice{m}_p\|_{\infty}$,
\item[(ii)] in each nodal region the map $r\mapsto u\apice{m}_p(r)$ has exactly one critical point (which is either a local maximum or a local minimum point, and they alternate),
\item[(iii)] $\int_{B}|\nabla u\apice{m}_p(y)|^{2}dy = \int_{B}|u\apice{m}_p(y)|^{p+1}dy  \underset{p\rightarrow p_S} {\longrightarrow} mS_N^{\frac{N}{2}}$,
\end{itemize}
where $S_N$ is the best constant for the Sobolev embedding $H^1_0(B)\hookrightarrow L^{2^*}(B)$:
\begin{equation}\label{SobolevEmb}\sqrt{S_N}\|v\|_{L^{2^*}(B)}\leq \|\nabla v\|_{L^2(B)}, \ \ \forall v\in H^1_0(B).
\end{equation}
\end{proposition}

The statement (i)--(iii) are known, in particular (i) and (ii) follow by o.d.e. arguments. Instead (iii) derives by the uniqueness of $u\apice{m}_p$. In fact on the one hand  it is easy to see by the Sobolev embedding that  for each nodal region $B_p$ of $u\apice{m}_p$ we have
\begin{equation}\label{energyRegione}
\lim_{p\rightarrow p_S}\int_{B_p}|\nabla u\apice{m}_p(y)|^{2}dy \geq S_N^{\frac{N}{2}}.
\end{equation} On the other hand,  for any fixed $m\in\N^+$, radial nodal solutions of \eqref{problem} with $m$ nodal regions and whose energy converges to $mS_N^{\frac{N}{2}}$  have been obtained in \cite{PistoiaWeth}.

\

Now let us denote by $r\apice{m}_{i,p}$, $i=1,\ldots m-1$, the nodal radii of $u\apice{m}_p$ and for uniformity of notation, by $r\apice{m}_{m,p}$ the radius of $B$. Then writing with abuse of notation $u\apice{m}_p(r)=u\apice{m}_p(|x|)$, we have
\begin{equation} \label{rp}
\begin{array}{lr}
0<r\apice{m}_{1,p}<r\apice{m}_{2,p}<\cdots <r\apice{m}_{m-1,p}<r\apice{m}_{m,p}:=1\\
u\apice{m}_p(r\apice{m}_{i,p})=0\qquad\mbox{i=1,\ldots,m.}
\end{array}
\end{equation} 
Moreover we denote by $s\apice{m}_{i,p}$, $i=0,\ldots m-1$, the unique maximum point of $|u\apice{m}_p|$ in each nodal region, so 
\begin{equation} \label{sp}
\begin{array}{lr}
s\apice{m}_{0,p}=0\\
s\apice{m}_{i,p} \in (r\apice{m}_{i,p}, r\apice{m}_{i+1,p}),\quad i=1,\ldots, m-1  \ (\mbox{if }m\geq 2)
\end{array}
\end{equation} 
and
\[(u\apice{m}_p)'(s\apice{m}_{i,p})=0, \quad i=0,\ldots, m-1.\]
Let us denote the   $m$ nodal regions of $u\apice{m}_p$ by $B\apice{m}_{i,p}\subset\R^N,$ $i=0,\ldots, m-1$, namely:
\begin{equation}\label{regioniNodaliupNotazione}
\begin{array}{lr}
B\apice{m}_{0,p}:=\{x\in\mathbb R^N\ :\ |x|< r\apice{m}_{1,p}\}\\
B\apice{m}_{i,p}:=\{x\in\mathbb R^N\ :\ r\apice{m}_{i,p}< |x|< r\apice{m}_{i+1,p}\}, \quad i=1,\ldots, m-1 \ (\mbox{if }m\geq 2).
\end{array}
\end{equation}
Then we consider the restriction of $|u\apice{m}_p|$ to the $i$-th nodal region
\begin{equation}\label{defNodalRegionsm}
u\apice{m}_{i,p}:=|u\apice{m}_p|\chi_{B\apice{m}_{i,p}}, \quad i=0,\ldots, m-1.
\end{equation}
 and let us define   
\begin{equation}
M\apice{m}_{i,p}:=\|u\apice{m}_{i,p}\|_{\infty}=u\apice{m}_{i,p}(s\apice{m}_{i,p})=|u\apice{m}_{p}(s\apice{m}_{i,p})|, \quad i=0,\ldots, m-1.\label{Mpm}
\end{equation}
Observe that when $m=2$ then $u\apice{2}_{0,p}$ and $u\apice{2}_{1,p}$ are respectively the positive and negative part of $u\apice{2}_p$. 
\vspace{-0.4cm}
\begin{figure}[H]   \label{Figure}
  \footnotesize\centering
  \resizebox{450pt}{!}{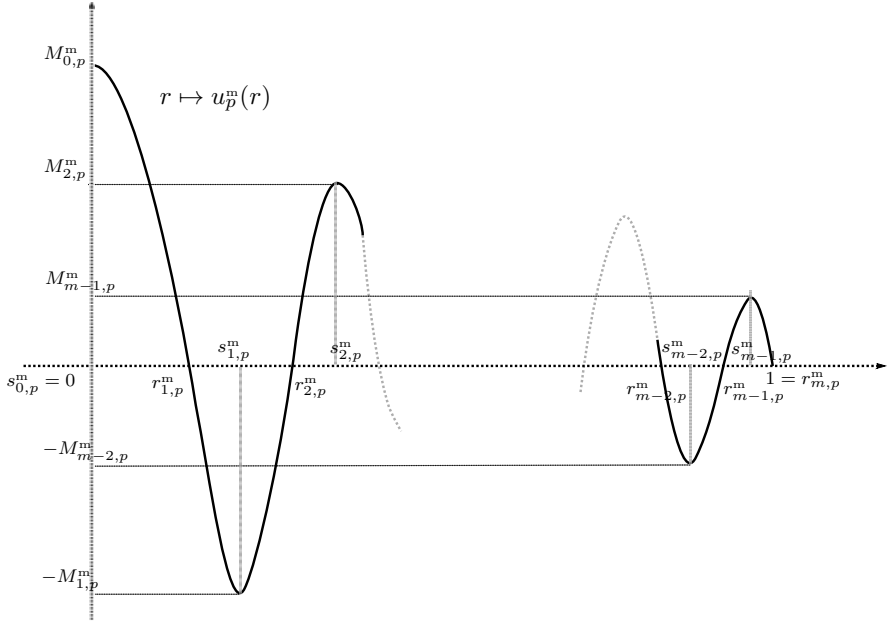}
 \caption{The radial solution  of \eqref{problem} having $m$ nodal regions}
\end{figure}

Our next result establishes the relation among  nodal radii in \eqref{rp},  maximum points in \eqref{sp} and scaling parameters in \eqref{Mpm} related to radial solutions of \eqref{problem} with a different number of nodal regions, $m$ and $h$ respectively:
\begin{lemma} \label{lemma:legami}
Let $m\in\mathbb N^+$, $m\geq 2$ and $h= 1, \ldots, m-1$. Then for $j=1,\ldots, h$ we have:
\begin{eqnarray}
 \label{SecondoLegame}
&& r\apice{\emph h}_{j,p}=\frac{r\apice{\emph m}_{j,p}}{r\apice{\emph m}_{h,p}}.
\end{eqnarray}
Moreover for $j=0,\ldots, h-1$ we have:
\begin{eqnarray}
&& s\apice{\emph h}_{j,p}=\frac{s\apice{\emph m}_{j,p}}{r\apice{\emph m}_{h,p}}
\label{PrimoLegame}
\\
\label{QuartoLegame}
&& s\apice{\emph h}_{j,p}(M\apice{\emph h}_{j,p})^{\frac{p-1}{2}} =s\apice{\emph m}_{j,p}( M\apice{\emph m}_{j,p})^{\frac{p-1}{2}}
\\
\label{TerzoLegame}
&& (M\apice{\emph h}_{j,p})^{\frac{p-1}{2}} =r\apice{\emph m}_{h,p}( M\apice{\emph m}_{j,p})^{\frac{p-1}{2}}.
\end{eqnarray}
\end{lemma}

\begin{proof}
Let $h= 1, \ldots, m-1$ and consider the restriction of the solution $u\apice{m}_{p}$ to the first $h$ nodal regions:
\begin{equation}
w\apice{m}_{h,p}:= u\apice{m}_{p}\chi_{ \bigcup_{n=0}^{h-1} B\apice{m}_{n,p}}.
\end{equation}
Then it is easy to check that the scaling $\widetilde{w}\apice{m}_{h,p}(|x|)$ of $w\apice{m}_{h,p}$ defined as
\begin{equation}
\widetilde{w}\apice{m}_{h,p}(|x|):= (r\apice{m}_{h,p})^{\frac{2}{p-1}}w\apice{m}_{h,p}(r\apice{m}_{h,p} |x|)
\end{equation}
is a radial solution to \eqref{problem} having $h$ nodal regions and such that $\widetilde{w}\apice{m}_{h,p}(0)>0$. By uniqueness
\begin{equation}\label{coincidono}
\widetilde{w}\apice{m}_{h,p}= u\apice{h}_{p}. 
\end{equation}
As a consequence we immediately get
\eqref{PrimoLegame} and \eqref{SecondoLegame}. Moreover we also have:
\[
M\apice{h}_{0,p}=u\apice{h}_p(0)\overset{\eqref{coincidono}}{=} 
(r\apice{m}_{h,p})^{\frac{2}{p-1}}u\apice{m}_{p}(0){=}
(r\apice{m}_{h,p})^{\frac{2}{p-1}}M\apice{m}_{0,p},
\]
which gives \eqref{TerzoLegame} in the case  $j=0$. Instead, when $j=1,\ldots, h-1$, we have:
\[
M\apice{h}_{j,p}=u\apice{h}(s\apice{h}_{j,p})\overset{\eqref{coincidono}}{=} 
(r\apice{m}_{h,p})^{\frac{2}{p-1}}u\apice{m}_{p}(r\apice{m}_{h,p}s\apice{h}_{j,p})\overset{\eqref{PrimoLegame}}{=}
(r\apice{m}_{h,p})^{\frac{2}{p-1}}u\apice{m}_{p}(s\apice{m}_{j,p})=(r\apice{m}_{h,p})^{\frac{2}{p-1}} M\apice{m}_{j,p},
\]
which ends the proof of  \eqref{TerzoLegame}. 
Last by \eqref{TerzoLegame} and \eqref{PrimoLegame} we get \eqref{QuartoLegame}.
\end{proof}

\

In the sequel, in order to make the reading more fluid,  
when there is no possibility of misunderstanding we may drop the dependence on $m$ in our notations, writing, for instance, simply
$u_{i,p}, r_{i,p}, M_{i,p}, \ldots$ instead of $u\apice{m}_{i,p}, r\apice{m}_{i,p}, M\apice{m}_{i,p}\ldots$.

\

Similarly as in \cite[Lemma 2.1]{BenAyedElMehdiPacellaAlmostCritical} (where the case $m=2$ is considered) we get
\begin{proposition} \label{propositionComportamRegioNod} Let $m\in\mathbb N^+$. As $p\rightarrow p_S$ we have, for any $i=0,\ldots, m-1$:
\begin{eqnarray}
&& \label{limiteMezzaNormaGrad}
\int_{B}|\nabla u\apice{\emph m}_{i,p}(y)|^{2}dy= \int_{B}| u\apice{\emph m}_{i,p}(y)|^{p+1}dy \longrightarrow S_N^{\frac{N}{2}}
\\
\label{limiteMezzaNormaStar}
&&\int_{B}|u\apice{\emph m}_{i,p}(y)|^{2^*}dy \longrightarrow S_N^{\frac{N}{2}}
\\
&&\label{limiteMezzaNormap}
\int_{B}|u\apice{\emph m}_{i,p}(y)|^{\frac{N}{2}(p-1)} \longrightarrow S_N^{\frac{N}{2}}
\\
&&
\label{weakConv}
u\apice{\emph m}_{p}\rightharpoonup 0\mbox{ in }H^1_0(B)
\\
&&
\label{epsilonpmN3}
M\apice{\emph m}_{i,p}\longrightarrow +\infty
\end{eqnarray}
\end{proposition}

\begin{proof}
\eqref{limiteMezzaNormaGrad} is a direct consequence of  Proposition \ref{PropositionUnicoMaxeMin}-(iii) and \eqref{energyRegione}. The convergence results in  \eqref{limiteMezzaNormaStar} and  \eqref{limiteMezzaNormap} follow then from \eqref{limiteMezzaNormaGrad}, indeed:
\begin{eqnarray*}S_N^{\frac{N}{2}}&\overset{\eqref{limiteMezzaNormaGrad}}{=} &\lim_{p\rightarrow p_S} \frac{\left(\int_B|u\apice{m}_{i,p}(y)|^{p+1}dy\right)^{\frac{2^*}{(p+1)}}}{|B|^{\frac{2^*}{p+1}-1}}\overset{\mbox{\scriptsize{ H\"older}}}{\leq}
\lim_{p\rightarrow p_S}\int_{B}|u\apice{m}_{i,p}(y)|^{2^*}dy\\
&\overset{\eqref{SobolevEmb}}{\leq}& \lim_{p\rightarrow p_S}\frac{\|\nabla u\apice{m}_{i,p}\|_{L^{2}(B)}^{2^*}}{S_N^{\frac{N}{N-2}}}\overset{\eqref{limiteMezzaNormaGrad}}{=}S_N^{\frac{N}{2}},\end{eqnarray*}
which proves \eqref{limiteMezzaNormaStar} and similarly we get \eqref{limiteMezzaNormap}:
\begin{align*} S_N^{\frac{N}{2}}& \overset{\scriptsize{\begin{array}{cr}\eqref{limiteMezzaNormaGrad}\\\eqref{limiteMezzaNormaStar}\end{array}}}{=}\lim_{p\rightarrow p_S} \frac{\left(\int_B|u\apice{m}_{i,p}(y)|^{p+1}dy\right)^{\frac{N}{2}}}{\left(\int_B|u\apice{m}_{i,p}(y)|^{2^*}dy\right)^{\frac{N-2}{2}}}\overset{\mbox{\scriptsize{ H\"older}}}
{\leq}
\lim_{p\rightarrow p_S}\int_{B}|u\apice{m}_{i,p}(y)|^{\frac{N}{2}(p-1)}dy\\
& \overset{\mbox{\scriptsize{ H\"older}}}{\leq} \lim_{p\rightarrow p_S}\int_B|u\apice{m}_{i,p}(y)|^{p+1}dy\overset{\eqref{limiteMezzaNormaGrad}}{=}S_N^{\frac{N}{2}}.
\end{align*}
The proof of \eqref{weakConv} follows immediately by the fact that $(u\apice{m}_{i,p})_p$ is (by \eqref{limiteMezzaNormaGrad} and \eqref{limiteMezzaNormaStar}) a minimizing sequence for the Sobolev embedding $H^1_0(B)\hookrightarrow L^{2^*}(B)$,  so that $u\apice{m}_{i,p}\rightharpoonup 0$ in $H^1_0(B)$ as $p\rightarrow p_S$.\\
Finally the proof of \eqref{epsilonpmN3} follows by \eqref{limiteMezzaNormaStar} and \eqref{weakConv}, indeed fixing $\alpha\in (0,2^*)$, then as $p\rightarrow p_S$:
\[S_N^{\frac{N}{2}}\overset{\eqref{limiteMezzaNormaStar}}{\longleftarrow}\int_B|u\apice{m}_{i,p}|^{2^*}dy\leq \|u\apice{m}_{i,p}\|_{\infty}^{\alpha}\int_B|u\apice{m}_{i,p}|^{2^*-\alpha}dy,\]
so, since by \eqref{weakConv} and Rellich Theorem 
$\int_B|u\apice{m}_{i,p}|^{2^*-\alpha}\rightarrow 0\ \mbox{ as }p\rightarrow p_S$, then necessarily \eqref{epsilonpmN3} holds.
\end{proof}

\

\

We recall now the classical inequality due to  Strauss (\cite{Strauss}), which holds for any $v\in H_{rad}^{1}(\mathbb R^N)$, $N\geq 3$:
\begin{equation}\label{Strauss}|v(x)|\leq C_N\frac{\|\nabla v\|_{L^2(\mathbb R^N)}}{|x|^{\frac{N-1}{2}}}\qquad\mbox{ for any }x\neq 0,
\end{equation}
where $C_N>0$ is a constant independent of $v$. From it we easily deduce:
\begin{proposition}\label{Proposition:spVaAZero} Let $m\in\mathbb N^+$, $m\geq 2$. For $i=1,\ldots, m-1$ we have 
\[s\apice{\emph m}_{i,p}\rightarrow 0 \mbox{\ (and so also $r\apice{\emph m}_{i,p}\rightarrow 0$) \qquad as }\ p\rightarrow p_S\]
\end{proposition}

\begin{proof} Since $s\apice{m}_{i-1,p}<r\apice{m}_{i,p}<s\apice{m}_{i,p}$, it is enough to show the result in the case $i=m-1$. So setting $s_p:=s\apice{m}_{m-1,p}$, we want to prove that $s_p\rightarrow 0$ as $p\rightarrow p_S$. 
If by contradiction  $s_{p_n}\geq \alpha>0$  for a sequence $p_n\rightarrow p_S$ as $n\rightarrow +\infty$, then by  \eqref{Strauss} and Proposition  \ref{PropositionUnicoMaxeMin}-(iii) 
\[M\apice{m}_{m-1,p_n}=|u\apice{m}_p(s_{p_n})|\leq \frac{C_N}{|\alpha|^{\frac{N-1}{2}}}\|\nabla u\apice{m}_{p_n}\|_{L^2(\mathbb R^N)}\longrightarrow \frac{C_N}{|\alpha|^{\frac{N-1}{2}}} m S_N^{\frac{N}{2}}\ \mbox{ as }n\rightarrow +\infty.\]
So the sequence $(M\apice{m}_{m-1,p_n})_n$ would be bounded in contradiction with \eqref{epsilonpmN3}.
\end{proof}

\

The next propositions contain  crucial estimates for $|u\apice{m}_p|$ in each nodal region $ B\apice{m}_{i,p}$, $i=0,\ldots, m-1$.

\begin{proposition} \label{StimaFondamentalePositiva}
Let $m\in\N^+$, then
\begin{equation}\label{bound parte positivaFONDAMENTALE}
|u\apice{\emph m}_{p}(x)|\leq \frac{M\apice{\emph m}_{0,p} }{\left[ 1+ \frac{( M\apice{\emph m}_{0,p} )^{p-1}}{N (N-2)} |x|^2 \right]^{\frac{N-2}{2}}} \ \quad\forall\; x\in B\apice{\emph m}_{0,p}.
\end{equation}
where $B\apice{\emph m}_{0,p}\subset\R^N$ is as in \eqref{regioniNodaliupNotazione} and $M\apice{\emph m}_{0,p}>0$ as in \eqref{Mpm}.
\end{proposition}

\begin{proof}
The ordinary differential equation satisfied by $u\apice{m}_p$ can be turned by a suitable change of variable into an Emden-Fowler equation. Then the proof can be derived adapting the arguments contained in the papers \cite{AtkPel1,AtkPel2} of Atkinson and Peletier, who dealt with the Brezis-Nirenberg problem. Since the proof of the next Proposition \ref{StimaFondamentaleNegativa} is similar but slightly more involved, we refer to it for the details.
\end{proof}

\

\

Next, if $u\apice{m}_p$ changes sign (i.e. $m\geq 2$) we can estimate $|u\apice{m}_p|$ in a similar way  in  suitable proper subsets   $C\apice{m}_{i, p}\subset B\apice{m}_{i,p}$, $i=1,\ldots,m-1$. 
As one can see from the statement below,
when $i=1,\dots, m-2$ ($m\geq 3$), we make the assumption \eqref{Rmi},
which will be shown in Corollary \ref{cor:RmiSatisfied} to be always satisfied.

\begin{proposition} \label{StimaFondamentaleNegativa} Let  $\alpha\in (0, \frac{N-2}{2})$, $m\in\N^+$, $m\geq 2$ and  $i\in\{1,\ldots, m-1\}$. If   $m\geq 3$ assume  that 
\begin{equation}
\tag{$\mathcal{R}^{m}_i$}\label{Rmi}
\frac{s\apice{\emph m}_{i,p}}{r\apice{\emph m}_{i+1,p}}\longrightarrow 0 \qquad \mbox{ as }
p\rightarrow p_S, \ \ \forall i\neq m-1.
\end{equation} Then there exists $\gamma=\gamma(\alpha,m)\in (0,1)$, $\gamma(\alpha,m)\rightarrow 1$ as $\alpha\rightarrow 0$  and $\delta_i=\delta_i(\alpha,m)\in (0,\frac{4}{N-2})$ such that for  $p\geq p_S-\delta_i$  we have
\begin{equation}\label{bound parte negativaFONDAMENTALE}
|u\apice{\emph{m}}_{p}(x)|\leq \frac{M\apice{\emph{m}}_{i,p} }{\left[ 1+ \frac{2\alpha}{N (N-2)^2 } ( M\apice{\emph{m}}_{i,p})^{p-1}|x|^2 \right]^{\frac{N-2}{2}}} \ \quad\forall\; x\in C\apice{\emph{m}}_{i,p},   
\end{equation}
where 
\[ C\apice{\emph{m}}_{i,p} := \left\{ x\in \mathbb R^N : \ \gamma^{-\frac{1}{N}}s\apice{\emph{m}}_{i,p} <|x|< r\apice{\emph{m}}_{i+1,p}\right\}\, (\subset B\apice{\emph{m}}_{i,p})
\]
and $M\apice{\emph{m}}_{i,p}>0$ is defined in \eqref{Mpm}.
\end{proposition}
\begin{proof}
We argue as in \cite{Iacopetti}. Since $u_p$ is a radial solution to \eqref{problem} and $s_{i,p}$ is a critical point for it then $u_{i,p}=|u_p|\chi_{B_{i,p}}$ satisfies in particular
\begin{equation}
\left\{
\begin{array}{lr}
u_{i,p}''(r)+ \frac{N-1}{r} u_{i,p}'(r)+(u_{i,p}(r))^p=0\ & \ r \in (s_{i,p}, r_{i+1,p})\\
u_{i,p}'(s_{i,p})=0\\
u_{i,p}(r_{i+1,p})=0\\
u_{i,p}(s_{i,p})=M_{i,p}
\end{array}
\right.
\end{equation}
Let \[t:=\left(\frac{N-2}{r}\right)^{N-2}\]
and
\[y_p(t):=u_{i,p}\left(\frac{N-2}{t^{\frac{1}{N-2}}}\right)\]
then $y_p$ satisfies an Emden-Fowler type ordinary differential equation:
\begin{equation}
\left\{
\begin{array}{lr}
y_p''(t)+ t^{-k}(y_p(t))^p=0,\ & \ t \in (t_{1,p}, t_{2,p})\\
y_p'(t_{2,p})=0\\
y_p(t_{1,p})=0\\
y_p(t_{2,p})=M_{i,p}
\end{array}
\right.
\end{equation}
where
$k:=2\frac{N-1}{N-2}$, $t_{1,p}:= \left(\frac{N-2}{r_{i+1,p}}\right)^{N-2}$, $t_{2,p}:= \left( \frac{N-2}{s_{i,p}}\right)^{N-2}$ 
(notice that $y_p$, $t_{1,p}$ and $t_{2,p}$ depend also on $i$ but we have omitted it in the notations for simplicity).

\

\

\emph{STEP 1. We show that
\begin{equation}\label{deri} \left(y_p't^{k-1}y_p^{1-k}\right)' + t^{k-2}y_p^{-k} t_{2,p}^{1-k}(y_p(t_{2,p}))^{p+1}\leq 0, \mbox{ for all }\ t\in (t_{1,p}, t_{2,p})
\end{equation} }

\

\emph{Proof of STEP 1.} We differentiate $y_p't^{k-1}y_p^{1-k}$ and using $y_p''+ t^{-k}y_p^p=0$ we get
\begin{eqnarray*}
 \left(y_p't^{k-1}y_p^{1-k}\right)' & = &
 y_p''t^{k-1}y_p^{1-k}+y_p' (k-1) t^{k-2}y_p^{1-k}- (k-1)(y_p')^2
t^{k-1}y_p^{-k}
\\
&=& -t^{-1}y_p^{p+1-k}+y_p' (k-1) t^{k-2}y_p^{1-k}- (k-1)(y_p')^2
t^{k-1}y_p^{-k}
\\
&=& -2(k-1)t^{k-2}y_p^{-k}\left(\frac{1}{2(k-1)} t^{1-k}y_p^{p+1}  - \frac{1}{2} y_p'  y_p + \frac{1}{2}(y_p')^2
t\right)
\end{eqnarray*}
Adding and subtracting $t^{k-2}y_p^{-k}t_{2,p}^{1-k}(y_p(t_{2,p}))^{p+1}$ we deduce
\[
 \left(y_p't^{k-1}y_p^{1-k}\right)' + t^{k-2}y_p^{-k}t_{2,p}^{1-k}(y_p(t_{2,p}))^{p+1} =  -2 (k-1)t^{k-2}y_p^{-k} L_p(t) 
\]
where
\[L_p(t):= \frac{1}{2(k-1)} t^{1-k}y_p^{p+1}  - \frac{1}{2} y_p'  y_p + \frac{1}{2}(y_p')^2
t- \frac{1}{2(k-1)}t_{2,p}^{1-k}(y_p(t_{2,p}))^{p+1}\]
Hence \eqref{deri} is proved if we show that 
\begin{equation}\label{Lppositiva}
L_p(t)\geq 0\ \mbox{ for all }\  t\in  (t_{1,p}, t_{2,p}),
\end{equation}
which follows just observing that by definition $L_p(t_{2,p})=0$ and that $L_p'(t)\leq 0$ for $ t\in  (t_{1,p}, t_{2,p})$. Indeed by easy computations
\[L_p'(t)=\frac{p(N-2)-(N+2)}{2N}t^{1-k}y_p'(t)(y_p(t))^p\]
where $\frac{p(N-2)-(N+2)}{2N}<0$ (since $p<p_S$), $y_p(t)>0$ and $y_p'(t)\geq 0$ for $ t\in  (t_{1,p}, t_{2,p})$ (because $(u_{i,p})'(s)\leq 0$ for $s\in (s_{i,p},r_{i+1,p})$).

\

\

\emph{STEP 2. We show that for any $\alpha\in (0,\frac{N-2}{2})$ there exist $\gamma=\gamma(\alpha)\in (0,1)$, $\delta_i=\delta_i(\alpha)>0$ such that
\begin{equation}
\label{stimasuyp}
y_p(t)\leq  M_{i,p}\left[ 1+ \frac{2}{N}(M_{i,p})^{p-1}t^{-\frac{2}{N-2}}\alpha \right]^{-\frac{N-2}{2}}, \  \mbox{ for }t \in (t_{1,p},\gamma^{\frac{N-2}{N}}t_{2,p}),\ p_S-p<\delta_i.
\end{equation}
}

\

\emph{Proof of STEP 2.} We integrate \eqref{deri} between $t$ and $t_{2,p}$ for all $t\in (t_{1,p}, t_{2,p})$. Since  $y_p'(t_{2,p})=0$ and $y_p(t_{2,p})=M_{i,p}$ we get
\[  y_p'(t)t^{k-1}y_p(t)^{1-k}\  \geq\  t_{2,p}^{1-k}(M_{i,p})^{p+1} \int_t^{t_{2,p}} s^{k-2}y_p(s)^{-k}ds \ \ \mbox{ for all }\ t\in (t_{1,p}, t_{2,p}).
\]
Since  $u_{i,p}\leq M_{i,p}$ by definition, it follows $y_p^{-k}\geq (M_{i,p})^{-k}$, so
\begin{eqnarray*} 
 y_p'(t)t^{k-1}y_p(t)^{1-k} &\geq &  t_{2,p}^{1-k}(M_{i,p})^{p+1-k} \int_t^{t_{2,p}} s^{k-2}ds
 \\
&= &  \frac{(M_{i,p})^{p+1-k}}{k-1}\left(1-\left( \frac{t}{t_{2,p}} \right)^{k-1}\right).
\end{eqnarray*}
Multiplying both side by $t^{1-k}$ we get
\[ 
 \frac{1}{2-k}(y_p(t)^{2-k})'=y_p'(t)y_p(t)^{1-k}\ \geq \ \frac{(M_{i,p})^{p+1-k}}{k-1}\left(t^{1-k}-\left( \frac{1}{t_{2,p}} \right)^{k-1}\right).
\]
Integrating between $t$ and $t_{2,p}$ and recalling that $y_p(t_{2,p})=M_{i,p}$, we have
\begin{eqnarray}
\frac{y_p(t)^{2-k}}{k-2}- \frac{(M_{i,p})^{2-k}}{k-2}& \geq & \frac{(M_{i,p})^{p+1-k}}{k-1}\left(-\frac{t_{2,p}^{2-k}}{k-2}+\frac{t^{2-k}}{k-2}- \frac{1}{t_{2,p}^{k-2}}+\frac{t}{t_{2,p}^{k-1}}\right)
\nonumber
\\
&= & \frac{(M_{i,p})^{p+1-k}}{k-1} t^{2-k}g\left(\left(\frac{t}{t_{2,p}}\right)^{k-1}\right),\label{inte}
\end{eqnarray}
where 
\[g(s):= \frac{1}{k-2}+s - \frac{k-1}{k-2}s^{\frac{k-2}{k-1}}, \ \ s \in [0,1]. \]
Observe that  
\begin{eqnarray*}
&& g(0) = \frac{1}{k-2}=\frac{N-2}{2}>0\\
&& g(1)  =  0\\
&& g'(s) = 1-s^{-\frac{1}{k-1}}<0 \ \mbox{ in}\  (0,1).
\end{eqnarray*}
so  $g(s)>0$ for all $s\in (0,1)$. Moreover, if for any $\alpha\in (0, \frac{N-2}{2})$ there exists only one $\gamma=\gamma(\alpha)\in (0,1)$ such that $g(\gamma)=\alpha$, $g(s)>\alpha$ for all $s\in [0,\gamma)$ and $\gamma\rightarrow 1$ as $\alpha\rightarrow 0$.
\\
Now remembering that in \eqref{inte} $s:=\left(\frac{t}{t_{2,p}}\right)^{k-1}$, it follows that
$s<\gamma$ if and only if $ t<\gamma^{\frac{1}{k-1}}t_{2,p}$. Let us observe that  $ t_{1,p}<\gamma^{\frac{1}{k-1}}t_{2,p}$ if and only if \[s_{i,p}^{N-2}<\gamma^{\frac{1}{k-1}} r_{i+1,p},\] which holds true, for any fixed $i\in\{1,\ldots, m-1\}$, if $p_S-p<\delta_i$, for some number $\delta_i(\gamma)>0$. In fact in the case $i=m-1$  we have, by definition, that $r_{i+1,p}\equiv 1$ so that the inequality follows directly from Proposition \ref{Proposition:spVaAZero}, while when $i\neq m-1$ it follows by the  assumption \eqref{Rmi}.\\
Hence from \eqref{inte} we have
\[
y_p(t)^{2-k}- (M_{i,p})^{2-k} \geq  
\frac{(M_{i,p})^{p+1-k}(k-2)}{k-1} t^{2-k} \alpha,\ \ \mbox{ for }t \in (t_{1,p},\gamma^{\frac{1}{k-1}}t_{2,p}),\ p_S-p<\delta_i
\]
which gives \eqref{stimasuyp}.

\

\

\emph{STEP 3. Estimate for $u_{i,p}$.}

\

\emph{Proof of STEP 3.} By definition we have $y_p(t)=u_{i,p}\left(\frac{N-2}{t^{\frac{1}{N-2}}}\right)$, so
by \eqref{stimasuyp} 
\[
u_{i,p}\left(\frac{N-2}{t^{\frac{1}{N-2}}}\right)\leq  M_{i,p}\left[ 1+ \frac{2}{N}(M_{i,p})^{p-1}t^{-\frac{2}{N-2}}\alpha \right]^{-\frac{N-2}{2}} 
\]
for $t \in (t_{1,p},\gamma^{\frac{N-2}{N}}t_{2,p})$, $p_S-p<\delta_i$. The conclusion follows for $|x|=r:=\frac{N-2}{t^{\frac{1}{N-2}}}$.
\end{proof}

\

\

We consider now, for $m\in\N^+$, the  $ m$ \emph{tail sets}
\begin{equation}\label{C1pm}
T\apice{m}_{i,p}:=\bigcup_{j=i}^{m-1} B\apice{m}_{j,p},\qquad i=0, \ldots, m-1
\end{equation}
where $B\apice{m}_{i,p}$ are the \emph{nodal regions} of $u\apice{m}_p$ defined in \eqref{regioniNodaliupNotazione} (observe that $T\apice{m}_{0,p}=B$, $T\apice{m}_{1,p}=B\setminus B\apice{m}_{0,p} $, \ldots, $T\apice{m}_{m-1,p}=B\apice{m}_{m-1,p}$).
We define the $m$ rescaled functions
\begin{eqnarray}\label{zeta}
&&z\apice{m}_{i,p}(x):=\frac{1}{M\apice{m}_{i,p}}\ u\apice{m}_p\Big(\frac{|x|}{(M\apice{m}_{i,p})^{\frac{p-1}{2}}}\Big), \quad x\in \widetilde{T}\apice{m}_{i,p}:=(M\apice{m}_{i,p})^{\frac{p-1}{2}}T\apice{m}_{i,p},\\\nonumber  && i=0, \ldots, m-1
\end{eqnarray}
 which are radial, solve 
\begin{equation}\label{equazionizp}
\left\{
\begin{array}{lr}
-\Delta z\apice{m}_{i,p}=|z\apice{m}_{i,p})|^{p-1}z\apice{m}_{i,p} &\ \mbox{ in }\widetilde{T}\apice{m}_{i,p}\\
z\apice{m}_{i,p}=0 &\mbox{ on }\partial( \widetilde{T}\apice{m}_{i,p})\\
z\apice{m}_{i,p}(s\apice{m}_{i,p})=1\  \mbox{ and }\ (z\apice{m}_{i,p})'(s\apice{m}_{i,p})=0
\end{array}
\right.
\end{equation}
and moreover, by the assumption \eqref{MaxInZero}, satisfy
\begin{equation}\label{zippositive}(-1)^iz\apice{m}_{i,p}>0  \quad \mbox{ in }\widetilde{B}\apice{m}_{i,p}:=(M\apice{m}_{i,p})^{\frac{p-1}{2}}B\apice{m}_{i,p}.
\end{equation}
 The main result of this section consists in proving that they all converge, up to the sign,  to the same function
\begin{equation}\label{UNgeq3}
U(x):=\left( \frac{N(N-2)}{N(N-2)+|x|^2} \right)^{\frac{N-2}{2}},
\end{equation}
which is the unique positive bounded radial solution to the critical equation in $\R^N$:
\begin{equation}
\label{criticalLimitEquation}
\left\{
\begin{array}{lr}
-\Delta U=U^{p_S}\ \mbox{ in }\mathbb R^N\\
U(0)=1
\end{array}
\right.
\end{equation}
and satisfies
\begin{equation}\label{normaU}
\int_{\R^N} |\nabla U|^2dx = \int_{\R^N} U^{2^*}dx= S_N^{\frac{N}{2}}.
\end{equation}

\

Precisely we show the following:
\begin{theorem}\label{theorem:analisiAsintoticaCasoNgeq3}
Let $m\in\N^+$. We have, as  $p\rightarrow p_S$:
\begin{eqnarray}
&&z\apice{\emph m}_{0,p}\longrightarrow U\quad\mbox{in $C^2_{loc}(\R^N)$},
\label{zppiuNg}
\\
&& (-1)^i z\apice{\emph m}_{i,p} \longrightarrow U\quad\mbox{in $C^2_{loc}(\R^N\setminus\{0\}),\  \forall\, i=1,\ldots, m-1\:$ (if $m\geq 2$).}
\label{zpmenoNg}
\end{eqnarray}
\end{theorem}

\

As we will see, in order to prove Theorem \ref{theorem:analisiAsintoticaCasoNgeq3} it is enough to scale each nodal region $B\apice{m}_{i,p}$ as
\begin{equation}\label{Btildeip}\widetilde{B}\apice{m}_{i,p}:=(M\apice{m}_{i,p})^{\frac{p-1}{2}}B\apice{m}_{i,p},\ \  i=0, \ldots, m-1
\end{equation} 
and show that the same result holds for the restriction of $z\apice{m}_{i,p}$ to the set $\widetilde{B}\apice{m}_{i,p}$, $i=0, \ldots, m-1$ (see Proposition \ref{Propos:zModificata} ahead).
We point out that the study of the  rescaled functions $z\apice{m}_{i,p}\chi_{\widetilde{B}\apice{m}_{i,p}}$, $i=1,\ldots, m-1,$ is more delicate as compared to the study of the first rescaled function $z\apice{m}_{0,p}\chi_{\widetilde{B}\apice{m}_{0,p}}$. The main reason is that the radius $s\apice{m}_{i,p}$, where the maximum of $|u\apice{m}_p|=|u\apice{m}_p(r)|$ is achieved  in the nodal region $B\apice{m}_{i,p}$, depends on $p$ when $i\neq 0$, while  $s\apice{m}_{0,p}\equiv 0$, for any $p$.

Moreover let us observe that also the nodal radii $r\apice{m}_{i,p}$ depend on $p$. When $i=1,\dots, m-1$ we know by Proposition \ref{Proposition:spVaAZero} that both $r\apice{m}_{i,p}$ and  $s\apice{m}_{i,p}$ converge to zero as $p\rightarrow p_S$ and, before proving Theorem \ref{theorem:analisiAsintoticaCasoNgeq3},  we need to get precise information about their rate of convergence. In particular in order to determine the limit problem we need to understand how $s\apice{m}_{i,p}$ and $r\apice{m}_{i,p}$ behave with respect to the rescaling parameters $(M\apice{m}_{i,p})^{\frac{p-1}{2}}$.

\

To this aim for $m\in\N^+$, $m\geq 2$ and $i=1,\ldots, m-1$, let us define the following properties:
\begin{align}
\tag{$\mathcal{A}^{m}_i$}\label{Ami}
r\apice{m}_{i,p} ( M\apice{m}_{i-1,p} )^{\frac{p-1}{2}}\longrightarrow +\infty \qquad \mbox{ as }
p\rightarrow p_S
\\
\tag{$\mathcal{B}^{m}_i$}\label{Bmi}
s\apice{m}_{i,p} ( M\apice{m}_{i,p} )^{\frac{p-1}{2}}\longrightarrow 0 \qquad \mbox{ as }
p\rightarrow p_S.
\end{align}
Clearly  \eqref{Bmi} implies
\begin{equation}
\tag{$\mathcal{C}^m_i$}\label{Cmi}
r\apice{m}_{i,p} ( M\apice{m}_{i,p} )^{\frac{p-1}{2}}\longrightarrow 0 \qquad \mbox{ as }
p\rightarrow p_S.
\end{equation}

\

We can easily prove that the first property holds, indeed we have:
\begin{proposition}\label{proposition:AmiValePerOgnim} 
Let $m\in\N^+$, $m\geq 2$. Then
\[(\mathcal{A}\apice{\emph m}_{i}) \mbox{ holds true }\qquad  \mbox{  for any }\  i=1,\ldots, m-1.\]
\end{proposition}
\begin{proof}
Let $i\in\{1,\ldots, m-1\}$, we want to show that
\[
r\apice{m}_{i,p} ( M\apice{m}_{i-1,p} )^{\frac{p-1}{2}}\longrightarrow +\infty \qquad \mbox{ as }
p\rightarrow p_S.
\]
This follows directly from Lemma \ref{lemma:legami} and Proposition \ref{propositionComportamRegioNod}. Indeed, choosing $h:=i$ and $j:=i-1$ into \eqref{TerzoLegame}  and using \eqref{epsilonpmN3},  we get:
\[r\apice{m}_{i,p}( M\apice{m}_{i-1,p})^{\frac{p-1}{2}} \overset{ \eqref{TerzoLegame}}{=}  (M\apice{i}_{i-1,p})^{\frac{p-1}{2}} \overset{\eqref{epsilonpmN3}}{\longrightarrow} +\infty\quad  \mbox{ as $p\rightarrow p_S$}.
\]
\end{proof}

\

Property \eqref{Bmi} is more difficult to be obtained. First we prove it for $i=m-1$ (Proposition \ref{prop:Bm-1ValePerOgnim} below)  and then we extend it to the remaining cases (Proposition \ref{prop:BmiValePerOgnim})  by means of Lemma \ref{lemma:legami}.

\


\begin{proposition}\label{prop:Bm-1ValePerOgnim}
Let $m\in\mathbb N^+$, $m\geq 2$. Then
\[(\mathcal{B}\apice{\emph m}_{m-1}) \mbox{ (and hence also $(\mathcal{C}\apice{\emph m}_{m-1})$) holds true }.\]
\end{proposition}

\

We first get the following easy estimate.

\begin{lemma} \label{lemma:stimaDerivataInrp}
There exists $C_N:=C_N(m)>0$  and $\delta=\delta (m) >0$ such that:
\[
\big|(u\apice{\emph m}_p)'(r)\big|\leq \frac{C_N}{r^{\frac{p+1}{p-1}}} \quad\  \forall\, r\in (0,1), \ \forall\, (0<) p_S-p\leq \delta.
\]
\end{lemma}
\begin{proof} Writing \eqref{problem} in polar coordinates it is easy to see that
\[\left( (u\apice{m}_p)'(r)\,r^{N-1} \right)'=-r^{N-1}|u\apice{m}_p(r)|^{p-1}u\apice{m}_p(r),\]
so integrating on $(0,r)$ (recall that $(u\apice{m}_p)'(0)=0$), by H\"older inequality, we have
\begin{eqnarray*}\big|(u\apice{m}_p)'(r)\big|\,r^{N-1}&\leq&\int_{\{|x|<r\}}|u\apice{m}_p(x)|^p \, dx\\
&\leq & \omega_N^{1-\frac{2p}{N(p-1)}}r^{N\left(1-\frac{2p}{N(p-1)}\right)} \left[\int_{B}|u\apice{m}_p(x)|^{\frac{N}{2}(p-1)} \, dx\right]^{\frac{2p}{N(p-1)}}
\end{eqnarray*}
and the conclusion follows from \eqref{limiteMezzaNormap}.
\end{proof}

\

\begin{proof}[Proof of Proposition \ref{prop:Bm-1ValePerOgnim}]
In order to shorten the notations let us set $s_p:=s\apice{m}_{m-1,p}$ and $M_p:=M\apice{m}_{m-1,p}$.
Hence to prove $(\mathcal{B}\apice{m}_{m-1})$ means to show that 
\begin{equation}\label{tesibh}
s_p(M_p)^{\frac{p-1}{2}} {\longrightarrow} \,0 \quad\mbox{ as }p\rightarrow p_S.
\end{equation}
We also set $r_p:= r\apice{m}_{m-1,p} $ and we define 
\begin{equation}\label{zpmenoNsimplify}  z_p:=z\apice{m}_{m-1,p},
\end{equation} where $z\apice{m}_{m-1,p}$  is the rescaled function defined in \eqref{zeta} for $i=m-1$, i.e. the one related to the last nodal region $T\apice{m}_{m-1,p}=B\apice{m}_{m-1,p}$. Recall (see \eqref{equazionizp} and \eqref{zippositive} with $i=m-1$) that it satisfies 
\begin{equation}\label{equazionizpProof}
\left\{
\begin{array}{lr}
-\Delta z_p=z_p^p &\ \mbox{ in }\widetilde{B}\apice{m}_{m-1,p}\\
z_p=0 &\mbox{ on }\partial( \widetilde{B}\apice{m}_{m.1,p})\\
z_p(s_p)=1\  \mbox{ and }\ (z_p)'(s_p)=0
\end{array}
\right.
\end{equation}
with $\widetilde{B}\apice{m}_{m-1,p}= \big\{r_p(M_p)^{\frac{p-2}{2}}<|x|< (M_p)^{\frac{p-2}{2}}\big\}$. Moreover $z_p$ does not change sign in $\widetilde{B}\apice{m}_{m-1,p}$ and w.l.g. let us assume that
\[z_p>0 \ \mbox{ in }\widetilde{B}\apice{m}_{m-1,p}.\]  We follow similar arguments as in the proofs of \cite[Lemma 4-5]{Iacopetti} (which concern the study of the least-energy nodal radial solution for the Brezis-Nirenberg problem) and consider also (setting $s:=|x|$) the one-dimensional rescaling of $u\apice{m}_{p}$:
\[w_p(s):=z_p\left(s+s_p(M_p)^{\frac{p-1}{2}}\right)=\frac{1}{M_p}\ u\apice{m}_p\Big(s_p +\frac{s}{(M_p)^{\frac{p-1}{2}}}\Big), \qquad s\in (a_p,b_p),\]
where 
\[\begin{array}{ll}
& a_p:=(r_p-s_p) (M_p)^{\frac{p-1}{2}},\\
& b_p:=(1-s_p) (M_p)^{\frac{p-1}{2}}.
\end{array}
\]
Then $w_p$ satisfies 
\begin{equation}\label{eq:w_p}
\left\{
\begin{array}{lr}
w_p''(s)+\displaystyle{\frac{N-1}{s+s_p(M_p)^{\frac{p-1}{2}}}}w_p'(s)+w_p(s)^{p}=0 \qquad s\in(a_p,b_p)\\
w_p'(0)=0,\ \ w_p(0)=1\\
w_p\geq 0 
\end{array}\right..
\end{equation}
Also let us observe that by Proposition \ref{Proposition:spVaAZero}  and \eqref{epsilonpmN3} one has that
\[b_p\rightarrow +\infty \  \ \mbox{ as }p\rightarrow p_S.\]

\

We divide the proof into two steps.

\

\emph{STEP 1. First we show that there exists $C>0$ independent of $p$ such that:
\begin{equation}
\label{noinfi}
s_p(M_p)^{\frac{p-1}{2}}\leq C.
\end{equation}}

\emph{Proof of STEP 1.} Assume by contradiction that  up to a subsequence
$s_p(M_p)^{\frac{p-1}{2}}\rightarrow +\infty$.\\
Up to a subsequence $a_p\rightarrow \bar a$, where $\bar a\in [-\infty, 0]$.\\
If $\bar a=-\infty$ or $\bar a<0$, then passing to the limit into \eqref{eq:w_p} we get that $w_p\rightarrow w$ in $C^1_{loc}(\bar a, +\infty)$ where $w$ solves the limit problem
\begin{equation}\label{eq:w}
\left\{
\begin{array}{lr}
w''(s)+w(s)^{p_S}=0 \qquad s\in(\bar a,+\infty)\\
w'(0)=0,\ \ w(0)=1\\
\end{array}\right.
\end{equation}
and so in particular, by definition of $w_p$, $w>0$ in $(\bar a,+\infty)$.
By a change of variable we have
\begin{eqnarray}\label{pippi}
\int_{\{r_p<|x|<1\}} |u\apice{m}_p(x)|^{\frac{N}{2}(p-1)}dx &=& \omega_N\int_{r_p}^1 |u\apice{m}_p(r)|^{\frac{N}{2}(p-1)} r^{N-1}dr \nonumber\\
&\geq &
 \omega_N s_p^{N-1}\int_{s_p}^1 |u\apice{m}_p(r)|^{\frac{N}{2}(p-1)} dr \nonumber \\
&=& \omega_N \left[s_p(M_p)^{\frac{p-1}{2}}\right]^{N-1}\int_{0}^{b_p} |w_p(s)|^{\frac{N}{2}(p-1)} ds,
\end{eqnarray}
and by Fatou's lemma 
\[\liminf_{p\rightarrow p_S}\int_{0}^{b_p} |w_p(s)|^{\frac{N}{2}(p-1)} ds \geq \int_{0}^{+\infty} |w(s)|^{2^*} ds >0. \]
Hence passing to the limit into \eqref{pippi} we get 
\[\lim_{p\rightarrow +\infty}\int_{B} |u\apice{m}_{m-1,p}(x)|^{\frac{N}{2}(p-1)}dx =\lim_{p\rightarrow +\infty}\int_{\{r_p<|x|<1\}} |u\apice{m}_{p}(x)|^{\frac{N}{2}(p-1)}dx =+\infty,\]
which is in contradiction with \eqref{limiteMezzaNormap}.

\

If $\bar a=0$ the previous argument fails because it could be $w\equiv 0$. So we consider the rescaled function $z_p$  in \eqref{zpmenoNsimplify}  which is uniformly bounded and  solves \eqref{equazionizpProof}.

By definition $z_p\big(r_p(M_p)^{\frac{p-1}{2}}\big)=0$ and $z_p\big(s_p(M_p)^{\frac{p-1}{2}}\big)=1$ for any $p\in (1,p_S)$, so
\[\frac{\left|z_p\big(s_p(M_p)^{\frac{p-1}{2}}\big)\ -\ z_p\big(r_p(M_p)^{\frac{p-1}{2}}\big)    \right|}{|a_p|}=\frac{1}{|a_p|}\rightarrow +\infty\quad\mbox{ as }p\rightarrow p_S.\]

where, since $z_p$ is regular, one has 
\[\frac{\left|z_p\big(s_p(M_p)^{\frac{p-1}{2}}\big)\ -\ z_p\big(r_p(M_p)^{\frac{p-1}{2}}\big)    \right|}{|a_p|}=|(z_p)'(\xi_p) |
\]
for some $\xi_p\in \left(r_p(M_p)^{\frac{p-1}{2}}, s_p(M_p)^{\frac{p-1}{2}} \right)$.
As a consequence
\begin{equation}\label{derivataInPuntoEsplode}
  |(z_p)'(\xi_p) | \rightarrow  +\infty\quad\mbox{ as }p\rightarrow p_S.
  \end{equation}
Since byProposition \ref{PropositionUnicoMaxeMin} we know that $(z_p)'>0$ in $\left(r_p(M_p)^{\frac{p-1}{2}}, s_p(M_p)^{\frac{p-1}{2}} \right)$ and  moreover by definition $z_p>0$, by writing the equation \eqref{equazionizpProof} in polar coordinates  it is easy to see that
\[(z_p)''<0\quad \mbox{ in }\left(r_p(M_p)^{\frac{p-1}{2}},s_p(M_p)^{\frac{p-1}{2}} \right),\]
hence by \eqref{derivataInPuntoEsplode} 
\begin{equation}\label{infiContra} (z_p)'
\big(
r_p(M_p)^{\frac{p-1}{2}}
\big) 
\geq (z_p)'(\xi_p)  \rightarrow  +\infty\quad\mbox{ as }p\rightarrow p_S.
\end{equation}
On the other side by Lemma \ref{lemma:stimaDerivataInrp} we also obtain
\begin{equation}\label{zeroContra}
\big|(z_p)'\big(
r_p(M_p)^{\frac{p-1}{2}}
\big) \big|\leq \frac{C_N}{\big(
r_p(M_p)^{\frac{p-1}{2}}
\big) ^{\frac{p+1}{p-1}}},
\end{equation}
where, since $\bar a=0$, then $
r_p(M_p)^{\frac{p-1}{2}}
\rightarrow +\infty$, and so \eqref{zeroContra} gives a contradiction with \eqref{infiContra}.

\

\emph{STEP 2. We show \eqref{tesibh}.}

\

\emph{Proof of STEP 2.} We argue by contradiction assuming by the results of \emph{STEP 1.} that,  up to a subsequence, $s_p(M_p)^{\frac{p-1}{2}}\rightarrow s_0>0$ as $p\rightarrow p_S$. Then, since $0<r_p<s_p$, we can have one of the following possibilities for $a_p$:
\begin{itemize}
\item[(i)] $a_p\rightarrow 0$
\item[(ii)] $a_p\rightarrow \bar a<0$.
\end{itemize}
Next we show that they both lead to a contradiction.

If we assume (i) we can repeat the same proof as in the case $\bar a=0$ in \emph{STEP 1.} The only difference is that now one has  $r_p(M_p)^{\frac{p-1}{2}}\rightarrow s_0$, which still implies a uniform bound of $(z_p)'(r_p(M_p)^{\frac{p-1}{2}})$ by \eqref{zeroContra}. This gives again a contradiction with \eqref{infiContra}.

Let us assume (ii) and define $r_0:=\bar{a}+s_0$. Clearly $r_0\in [0,s_0)$ and $r_p(M_p)^{\frac{p-1}{2}}\rightarrow r_0$.

If $r_0>0$, then we consider again the rescaled function $z_p$ in \eqref{zpmenoNsimplify} which is uniformly bounded and solves \eqref{equazionizpProof}.
So we get that $z_p\rightarrow z$ in $C^2_{loc}(\Pi_{r_0})$ as $p\rightarrow p_S$, where $\Pi_{r_0}:=\{y\in \mathbb R^N: \ |y|>r_0\}$ and passing to the limit into \eqref{equazionizpProof} ($s_0>r_0$), we have that  $z$ is a positive radial solution of 
\begin{equation}\label{eq:z}
\left\{
\begin{array}{lr}
-\Delta z= z^{p_S} \mbox{ in }\Pi_{r_0}\\
z'(s_0)=0,\ \ z(s_0)=1\\
\end{array}\right.
\end{equation}
In particular $z\not\equiv 0$.
Next we show that $z$ can be extended by continuity to zero on $\partial\Pi_{r_0}$, from which we get that $z\in H^1_0( \Pi_{r_0})$. In fact observe that $(z_p)'$ is uniformly bounded in $(r_p(M_p)^{\frac{p-1}{2}},s_p(M_p)^{\frac{p-1}{2}})$ by a constant $M$. This 	is because we know that $(z_p)'$ is monotone decreasing in $(r_p(M_p)^{\frac{p-1}{2}},s_p(M_p)^{\frac{p-1}{2}})$ and also, by \eqref{zeroContra}
 and $r_p(M_p)^{\frac{p-1}{2}}\rightarrow r_0>0$,  that  $(z_p)'(r_p(M_p)^{\frac{p-1}{2}})$ is uniformly bounded.
 As a consequence
 \[z_p(s)\leq M\left[s-r_p(M_p)^{\frac{p-1}{2}}\right], \quad s\in (r_p(M_p)^{\frac{p-1}{2}},s_p(M_p)^{\frac{p-1}{2}})\]
 and so, passing to the limit as $p\rightarrow p_S$ we get
 \[z(s)\leq M\left[s-r_0\right], \quad s\in (r_0,s_0),\]
 from which the extension property follows. 
 \\
Observe now that when $i=m-1$ the uniform upper bound \eqref{bound parte negativaFONDAMENTALE} for $u\apice{m}_p$ in Proposition \ref{StimaFondamentaleNegativa} holds (indeed let us recall that in the case $i=m-1$  the assumption \eqref{Rmi} is not required). By scaling it gives the following upper bound for $z_p$:
\[|z_p(y)|\leq\frac{1}{\left( 1+ \frac{2\alpha}{N (N-2)^2 }  |y|^2 \right)^{\frac{N-2}{2}}} \ \quad\forall y\in \widetilde C\apice{m}_{m-1, p},   \]
where 
\[ \widetilde C\apice{m}_{m-1,p} := \left\{ y\in \mathbb R^N : \ \gamma^{-\frac{1}{N}}s_p (M_p)^{\frac{p-1}{2}} <|y|< (M_p)^{\frac{p-1}{2}}\right\}\, \subset \widetilde B\apice{m}_{m-1,p}.
\]
Moreover $|z_p|\leq 1$ by definition, and so we get a uniform upper bound in the whole annulus $\widetilde B\apice{m}_{m-1,p}$, precisely:
\[|z_p(y)|\leq
\left\{
\begin{array}{ll}
1, & y\in \widetilde B\apice{m}_{m-1,p}\setminus\widetilde C\apice{m}_{m-1,p} \\
\frac{1}{\left( 1+ \frac{2\alpha}{N (N-2)^2 }  |y|^2 \right)^{\frac{N-2}{2}}}, &  y\in \widetilde C\apice{m}_{m-1,p} .
\end{array}\right.
  \]
Hence we can use Lebesgue's theorem to prove
 \begin{eqnarray}\label{best1}\int_{\Pi_{r_0}}|z|^{2^*}\,dx&\overset{\mbox{\tiny{Lebesgue}}}{ =} &
\lim_{p\rightarrow p_S}\int_{\widetilde C\apice{m}_{m-1,p}}| z_p|^{\frac{N}{2}(p-1)}\,dx 
\\
& = &
 \lim_{p\rightarrow p_S}\int_B |u\apice{m}_{m-1,p}|^{\frac{N}{2}(p-1)}\,dx\overset{\eqref{limiteMezzaNormap}}{=} S_N^{\frac{N}{2}}
 \end{eqnarray}
($\frac{N}{2}(p-1)\rightarrow 2^*$) and moreover, by Fatou's lemma  
\begin{eqnarray}\label{best2}\int_{\Pi_{r_0}}|\nabla z|^{2}\,dx&\overset{\mbox{\tiny{Fatou}}}{\leq }& \liminf_{p\rightarrow p_S}\int_{\widetilde C\apice{m}_{m-1,p}}|\nabla z_p|^{2}\,dx \\
&  =&  \liminf_{p\rightarrow p_S}
\frac{(M_p)^{\frac{N}{2}(p-1)}}{(M_p)^{p+1}}\int_{B}|\nabla u\apice{m}_{m-1,p}|^{2}\,dy\nonumber\\
&\leq &  \lim_{p\rightarrow p_S}\int_{B}|\nabla u\apice{m}_{m-1,p}|^{2}\,dy\overset{\eqref{limiteMezzaNormaGrad}}{=} S_N^{\frac{N}{2}},
\end{eqnarray}
where the last inequality follows from the fact that $\frac{N}{2}(p-1)\leq (p+1)$ for $p<p_S$ and $M_p >1$ definitely (indeed $M_p \rightarrow +\infty$ by \eqref{epsilonpmN3} with $i=m-1$).
As a consequence of  \eqref{best1} and \eqref{best2} the function $z$ attains the best Sobolev constant $S_N$ in $\Pi_{r_0}$ and this is clearly impossible since it is known that $S_N$ is not attained in domains strictly contained in $\R^N$. This concludes the proof in the case $r_0>0$.
\\
Assume now $r_0=0$, then $z_p\rightarrow z$ in $C^2_{loc}(\R^N\setminus\{0\})$ as $p\rightarrow p_S$, where $z$ is a radial, positive bounded solution to
\begin{equation}\label{eq:zaltra}
\left\{
\begin{array}{lr}
-\Delta z= z^{p_S} \mbox{ in }\R^N\setminus\{0\}\\
z'(s_0)=0
\end{array}
\right..
\end{equation}
Moreover by Fatou's lemma,  as in \eqref{best2}, we have
\begin{equation}
\label{asss}
\int_{\R^N}|\nabla z|^2\, dx< \infty.
\end{equation} Integrating 
$-\left( z'(r)r^{N-1} \right)' \overset{\eqref{eq:zaltra}}{=}z^{p_S}(r)r^{N-1}$  we get
\[0< \int_{\delta}^{s_0}z^{p_S}(r)r^{N-1}\,dr= z'(\delta)\delta^{N-1}\qquad \forall\delta\in (0,s_0),\]
where the left hand side is monotone decreasing in $\delta$ and so passing to the limit as $\delta\rightarrow 0^+$ we get
\[ z'(\delta)\delta^{N-1}\rightarrow \alpha>0,\]
namely $z'(r)\sim\frac{1}{r^{N-1}}$ around the origin and so
\[
\int_{\R^N}|\nabla z(x)|^2\, dx=\int_{0}^{+\infty}|z'(r)|^2r^{N-1}\, dr=+\infty,
\]
which contradicts \eqref{asss}.
\end{proof}

\
 
When $m\geq 3$ we need to prove property \eqref{Bmi} for the other indices $i\neq m-1$:

\begin{proposition}\label{prop:BmiValePerOgnim} Let $m\in\mathbb N^+$, $m\geq 3$. Then
\[\mbox{\eqref{Bmi}  (and hence also \eqref{Cmi}) holds true }\quad\forall\, i=1,\ldots, m-2.\]

\end{proposition}

\begin{proof}

Let us fix $i\in\{1\ldots, m-2\}$, we want to show that $s\apice{m}_{i,p} ( M\apice{m}_{i,p} )^{\frac{p-1}{2}}\longrightarrow 0$ as $p\rightarrow p_S$.\\
The proof  follows by  Lemma \ref{lemma:legami}  and Proposition \ref{prop:Bm-1ValePerOgnim}. Indeed choosing $j:=i$ and $h:=i+1$ into \eqref{QuartoLegame} we get
\[
s\apice{m}_{i,p}( M\apice{m}_{i,p})^{\frac{p-1}{2}}\overset{\eqref{QuartoLegame}}{=} s\apice{i+1}_{i,p}(M\apice{i+1}_{i,p})^{\frac{p-1}{2}}\overset{\mbox{\tiny{(Proposition \ref{prop:Bm-1ValePerOgnim})}}}{\longrightarrow } 0  \quad \mbox{ as }p\rightarrow p_S.
\]

\end{proof}

\

As a consequence of the properties \eqref{Ami} and \eqref{Bmi} we may remove the assumption \eqref{Rmi} in the statement of  Proposition \ref{StimaFondamentaleNegativa}, indeed:
\begin{corollary}\label{cor:RmiSatisfied}
Let $m\in\N^+$, $m\geq 3$. Then
\begin{equation}\label{RmiVale!}
\mbox{\eqref{Rmi} holds }\forall \,  i=1,\ldots, m-2\end{equation}
As a consequence the results in Proposition \ref{StimaFondamentaleNegativa} can be stated without the assumption  \eqref{Rmi}.
\end{corollary}
\begin{proof} By Proposition \ref{proposition:AmiValePerOgnim}, Proposition  \ref{prop:Bm-1ValePerOgnim}   and    Proposition \ref{prop:BmiValePerOgnim} we have that the properties \eqref{Ami} and \eqref{Bmi} are satisfied for any $i=1,\ldots, m-1$. Moreover observe that we haven't used \eqref{Rmi} in order to obtain them. Indeed \eqref{Rmi} appears only in the case $i\neq m-1$ of Proposition \ref{StimaFondamentaleNegativa} and, up to now, we have used the estimate \eqref{bound parte negativaFONDAMENTALE} of Proposition \ref{StimaFondamentaleNegativa} only   in the  proof of Proposition \ref{prop:Bm-1ValePerOgnim}, namely exactly in the case  $i=m-1$when the assumption \eqref{Rmi} is not needed to prove \eqref{bound parte negativaFONDAMENTALE}.

Last it is immediate to verify that
\[ (\mathcal{A}\apice{m}_{i+1}) \mbox{ and }(\mathcal{B}\apice{m}_i)\quad \Longrightarrow  \quad \eqref{Rmi}. 
\]
\end{proof}

\

\begin{remark}
Let us observe that  the rate of divergence of the $M\apice{\emph m}_{i,p}$ for different indexes $i$ cannot be the same, i.e.  it immediately follows from $(\mathcal{A}\apice{\emph m}_{i+1})$ and $(\mathcal{C}\apice{\emph m}_{i+1})$ that:
\begin{equation}\label{limiterappepsilon}
\frac{M\apice{\emph m}_{i,p}}{M\apice{\emph m}_{i+1,p}}\longrightarrow +\infty \  \mbox{ as }
p\rightarrow p_S,\ \mbox{
 $\forall\, i=0,\ldots, m-2$}.
\end{equation}
For nodal \emph{low-energy} solutions ($m=2$) of \eqref{problem}  with the points of maximum and minimum converging to the same point,  this was already known by the results in \cite[Theorem 1.2]{BenAyedElMehdiPacellaAlmostCritical}.

\end{remark}

\

\

Now, using the properties \eqref{Ami} and \eqref{Cmi} (which follows by \eqref{Bmi}), we can prove the following result, from which Theorem \ref{theorem:analisiAsintoticaCasoNgeq3} follows.

\

\begin{proposition}
\label{Propos:zModificata} 
Let $m\in\N^+$ and let
\[ \widetilde{B}\apice{\emph m}_{i,p}:=(M\apice{\emph m}_{i,p})^{\frac{p-1}{2}}B\apice{\emph m}_{i,p},\  i=0, \ldots, m-1
\]
where $B\apice{\emph m}_{i,p}$ are the nodal regions of $u\apice{\emph m}_p$ defined in \eqref{regioniNodaliupNotazione} and the parameters $M\apice{\emph m}_{i,p}>0$ are the ones introduced in \eqref{Mpm}. Then as  $p\rightarrow p_S$ we have:
\begin{eqnarray}
&&z\apice{\emph m}_{0,p}\chi_{\widetilde{B}\apice{\emph m}_{0,p}}\longrightarrow U\ \mbox{ in $C^2_{loc}(\R^N)$},
\label{zppiuN}
\\
&&(-1)^{i} z\apice{\emph m}_{i,p}\chi_{\widetilde{B}\apice{\emph m}_{i,p}} \longrightarrow U\ \mbox{ in $C^2_{loc}(\R^N\setminus\{0\}),\  \forall\, i=1,\ldots, m-1\:$ (if $m\geq 2$)}
\label{zpmenoN}
\end{eqnarray}
where the rescaled function $z\apice{\emph m}_{i,p}$ are defined in \eqref{zeta}.
\end{proposition}
\begin{proof}
The proof of \eqref{zppiuN} is standard. Indeed, since the functions $z\apice{m}_{0,p}$ are uniformly bounded, satisfy \eqref{equazionizp} in $\widetilde B\apice{m}_{0,p}$
and property $(\mathcal A\apice{m}_1)$ holds, we have that the limit of the domain $\widetilde B\apice{m}_{0,p}$ is the whole $\R^N$ and $z\apice{m}_{0,p}$ converge in $C^2_{loc}(\R^N)$ to a solution $z$ of \eqref{criticalLimitEquation}. The limit function $z$ has finite energy by Fatou's lemma, it is positive by \eqref{zippositive}  so it must necessarily be the function $U$ in \eqref{UNgeq3}.\\

Similarly we prove \eqref{zpmenoN}. Indeed  the rescaled functions $z\apice{m}_{i,p}$, $i=1,\ldots, m-1$, are uniformly bounded and solve  \eqref{equazionizp} in $\widetilde B\apice{m}_{i,p}$. 
The limit of the domains $\widetilde B\apice{m}_{i,p}$ is now $\R^N\setminus\{0\}$, this follows by the property $(\mathcal C\apice{m}_{m-1})$ in the case $i=m-1$ and 
by the properties 
$(\mathcal A\apice{m}_{i+1})$ and $(\mathcal C\apice{m}_i)$ in the other cases. By standard elliptic estimates, we have that $(-1)^iz\apice{m}_{i,p}\to z$ in $C^2_{loc}(\R^N\setminus\{0\})$ where $z$ is positive (by \eqref{zippositive}) radial, solves
\[
-\Delta z=z^{p_S}\qquad\mbox{in $\R^N\setminus\{0\}$}
\]
and (as for the previous case) has finite energy.\\
Exactly as in Lemma 6 and Lemma 7 of \cite{Iacopetti} we get that $z$ can be extended to a $C^1(\R^N)$ function such that $z(0)=1$, $\nabla z(0)=0$ and is a weak solution of \eqref{criticalLimitEquation} (in the whole $\R^N$). Hence $z$ must be the function $U$ of \eqref{UNgeq3}.
\end{proof}

\

\begin{proof}[Proof of Theorem \ref{theorem:analisiAsintoticaCasoNgeq3}]
The proof is similar to the one of Proposition  \ref{Propos:zModificata}. Just observe that $z\apice{m}_{i,p}$ is uniformly bounded in the whole rescaling of the tail set $\widetilde{T}\apice{m}_{i,p}$ in \eqref{zeta}, since it is uniformly bounded in $\widetilde{B}\apice{m}_{i,p}$ (as already observed in the proof of Proposition \ref{Propos:zModificata}) and moreover \eqref{limiterappepsilon} holds true. Observe also that the limit of the domain  $\widetilde{T}\apice{m}_{0,p}=(M\apice{m}_{0,p})^{\frac{p-1}{2}}B$ is clearly $\R^N$ (by \eqref{epsilonpmN3}), while  the limit of the domains  $\widetilde{T}\apice{m}_{i,p}$, when $i=1,\ldots, m-1$, is the set $\R^N\setminus \{0\}$  (by  \eqref{epsilonpmN3} and  property \eqref{Cmi}). The result then follows similarly as in the proof of Proposition \ref{Propos:zModificata}.
\end{proof}

\

We conclude the section with an estimate that will be important  throughout  the proof of Theorem \ref{teoPrincipaleMorse}: 
\begin{proposition}
Let $m\in\N^+$. There exist $\delta=\delta(m)>0$ and $C>0$ (independent of $m$) such that
\begin{equation}\label{Q3Nge3}
f\apice{\emph m}_p(|y|):=|y|^2|u\apice{\emph m}_p(y)|^{p-1}\leq C \qquad\mbox{ for any} \ y\in B \ \mbox{ and } \ \ p>p_S-\delta.
\end{equation}
\end{proposition}
\begin{proof}
\emph{Case I}:\quad $r:=|y|\in[0,r\apice{m}_{1,p}]$.\\
By Proposition \ref{StimaFondamentalePositiva} one has that $f\apice{m}_p(r)\leq \widetilde{g}_{p}\big(r(M\apice{m}_{0,p})^{\frac{p-1}{2}}\big)$, where for $s\in[0,+\infty)$
\[
\widetilde{g}_{p}(s):=\frac{s^2}{\left(1+\frac{1}{N(N-2)}s^2\right)^{\frac{(N-2)(p-1)}{2}}}.
\]
Since $\frac{(N-2)(p-1)}{2}\geq\frac32$ for $p$ sufficiently close to $p_S$, it can be easily seen that there exist $\delta>0$ and $C>0$ such that
\[
\widetilde{g}_{p}(s)\leq \frac{s^2}{\left(1+\frac{1}{N(N-2)}s^2\right)^{\frac{3}{2}}}\leq C\quad \mbox{ for any $s\in[0,+\infty)$ and $p>p_S-\delta$.}
\]
This concludes the proof of \emph{Case I}.

\

\emph{Case II}:\quad $r:=|y|\in(r\apice{m}_{i,p}, r\apice{m}_{i+1,p}]$, for some $i=1,\ldots, m-1$.\\
Let us fix $\alpha\in(0,\frac{N-2}{2})$ and consider $\gamma=\gamma(\alpha,m)$ defined in Proposition \ref{StimaFondamentaleNegativa}. Then for any $r\in(r\apice{m}_{i,p},\gamma^{-\frac{1}{N}}s\apice{m}_{i,p}]$ we use the property  \eqref{Bmi} (which is satisfied by Propositions \ref{prop:Bm-1ValePerOgnim}-\ref{prop:BmiValePerOgnim}) to prove that:
\[
f\apice{m}_p(r)\leq \gamma^{-\frac{2}{N}} (s\apice{m}_{i,p})^2|u\apice{m}_p(r)|^{p-1}\overset{\eqref{Mpm}}{\leq} \gamma^{-\frac{2}{N}} (s\apice{m}_{i,p})^2(M\apice{m}_{i,p})^{p-1}\overset{\eqref{Bmi}}{\underset{p\to p_S}{\longrightarrow}}0.
\]
Then clearly there exists  $C>0$ and there exists $\delta_i=\delta_{i}(m)>0$ such that $f\apice{m}_p(r)\leq C$, for any $r\in(r\apice{m}_{i,p},\gamma^{-\frac{1}{N}}s\apice{m}_{i,p}]$ and for any $p\geq p_S-\delta_i$.\\
For $r\in(\gamma^{-\frac{1}{N}}s\apice{m}_{i,p}, r\apice{m}_{i+1,p}]$ by Proposition \ref{StimaFondamentaleNegativa} and Corollary \ref{cor:RmiSatisfied}
\[
f\apice{m}_p(r)\leq \widehat{g}_p\big(r(M\apice{m}_{i,p})^{\frac{p-1}{2}}\big),
\]
where for $s\in[0,+\infty)$
\[
\widehat{g}_p(s):=\frac{s^2}{\left(1+\frac{2\alpha}{N(N-2)^2}s^2\right)^{\frac{(N-2)(p-1)}{2}}}.
\]
Exactly as in \emph{Case I}, fixing $\delta>0$ such that $\frac{(N-2)(p-1)}{2}\geq \frac32$ it turns out that 
\[
\widehat{g}_p(s)\leq C\qquad\quad\mbox{for any $s\in[0,+\infty)$ and $p>p_S-\delta$,}
\]
and this ends the proof of \emph{Case II}.
\end{proof}

\

\

\section{Approximations of eigenvalues and auxiliary weighted problems}\label{Section:Approximation}

In the following we summarize the construction and the results obtained in Sections $3$ and $4$ of \cite{DeMarchisIanniPacellaMathAnn}.
Along all the section $m\in N^+$ and $p\in (1,p_S)$ are fixed and $u\apice{m}_p$ is the  radial solution of \eqref{problem} having $m$ nodal regions, satisfying  the sign condition \eqref{MaxInZero} and already studied in the previous section.\\
\\
Let $L\apice{m}_{p}: H^2(B)\cap H^1_0(B)\rightarrow L^2(B)$ be the linearized operator at $u\apice{m}_p$, namely
\begin{equation}\label{linearizedOperator} L\apice{m}_{p}( v): =   -\Delta v-p|u\apice{m}_p(x)|^{p-1}v.
\end{equation}
The Dirichlet eigenvalues of $L\apice{m}_p$ in $B$, counted with their multiplicity, are
\[\begin{array}{ll} &\mu_1(m,p)< \mu_2(m,p)\leq\ldots\leq\mu_i(m,p)\leq\ldots,\\& \mu_i(m,p)\rightarrow +\infty\quad  \mbox{ as }i\rightarrow +\infty.\end{array}\]
Among these there are the \emph{radial} Dirichlet eigenvalues, which also form a sequence, denoted by:
\[\beta_i(m,p), \quad i\in\N^+.\]
As in Section \ref{sectionLowerBounMorse} the \emph{Morse index} of $u\apice{m}_p$ is denoted by $\mathsf{m}(u\apice{m}_p)$, while the \emph{radial Morse index} of $u\apice{m}_p$ (namely the number of negative radial eigenvalues of $L\apice{m}_p$) is denoted by $\mathsf{m_{rad}}(u\apice{m}_p)$.
\\
\\
By Theorem \ref{teo:AftPacGen} we know that
\begin{equation}\label{LemmaMorseIndexRadiale}
\mathsf{m}(u\apice{m}_p)\geq m+N(m-1) 
\end{equation}
and by  Theorem \ref{teoMorseRadiale} that 
\begin{equation}\label{LemmaMorseIndexRadiale2}
\mathsf{m_{rad}}(u\apice{m}_p)=m.
\end{equation} 

As in \cite{DeMarchisIanniPacellaMathAnn}, in order to compute the Morse index of $u\apice{m}_p$, we approximate the ball $B$ with the annuli:
\begin{equation}\label{def:anello}
A_n:=\{x\in \mathbb R^N\ :\ \frac{1}{n}<|x|<1 \}, \quad n\in \mathbb N^+,
\end{equation}
and we denote by \[{\mu_i^n}(m,p), \quad i\in\N^+\] the Dirichlet eigenvalues of $L\apice{m}_p$ in $A_n$ counted according to their multiplicity and 
by \[\beta_i^n(m,p), \quad i\in\N^+\] the radial Dirichlet eigenvalues of $L\apice{m}_p$ in $A_n$ counted with their multiplicity. Finally we denote by 
\begin{eqnarray}
& k^n_p(m) := \#\{\mbox{negative eigenvalues $\mu_i^n(m,p)$ of $L\apice{m}_p$ in $A_n$}\},
\\
& k^n_{p,rad}(m) := \#\{\mbox{negative radial eigenvalues $\beta_i^n(m,p)$ of $L\apice{m}_p$ in $A_n$}\}.
\end{eqnarray}
As proved in \cite{DeMarchisIanniPacellaMathAnn} (Lemma $3.2$ and Lemma $3.3$ therein) the following holds:
\begin{lemma}\label{lemma:morseProblemiSenzaPesoAnello}
For any fixed $m\in\N^+$ and any fixed $p\in (1,p_S)$ we have:
\begin{equation*}\label{eq:convEigenvalues}
\mu_i^n(m,p) \searrow \mu_i(m,p)\ \  \ \mbox{ and } \ \ \ \beta_i^n(m,p) \searrow \beta_i(m,p)\ \mbox{ as }\ n\rightarrow +\infty,\quad\forall\,i\in\N^+.
\end{equation*}
Hence there exists $n_p'=n_p'(m)\in \N^+$ such that
\[\mathsf{m}(u\apice{\emph m}_p)=k^n_p(m) \quad \mbox{ and }\quad \mathsf{m_{rad}}(u\apice{\emph m}_p)=k^n_{p,rad}(m),\ \mbox{ for }n\geq  n'_p.\]
\end{lemma}

In order to make a decomposition of the spectrum of $L\apice{m}_p$ we consider the auxiliary weighted linear operator $\widetilde{{L^n}}\apice{m}_{\!\!\!\! p}: H^2(A_n)\cap H^1_0(A_n)\rightarrow L^2(A_n)$ defined by:
\begin{equation}
\label{weightedOp}
\widetilde{{L^n}}\apice{m}_{\!\!\!\! p} (v): = |x|^2 \left(  -\Delta v-p|u\apice{m}_p(x)|^{p-1}v \right), \ \ x\in A_n,
\end{equation}
and  denote by
\[\widetilde{\mu_i^n}(m,p), \quad i\in\N^+\]
its eigenvalues  counted with their multiplicity. Observe that the corresponding eigenfunctions $h$ satisfy
\[\left\{
\begin{array}{lr}
-\Delta h(x)-p|u\apice{m}_p(x)|^{p-1}h(x)=\widetilde{\mu_i^n}(m,p)\, \frac{h(x)}{|x|^2} \ \ \ \ x\in A_n
\\
\\
h=0 \ \ \ \ \mbox{on } \partial A_n.
\end{array}
\right.
\]
Since $u\apice{m}_p$ is radial we also consider the following linear operator  $\widetilde{{L^n}}\apice{m}_{\!\!\!\!p,rad}: H^2((\frac{1}{n},1))\cap H^1_0((\frac{1}{n},1))\rightarrow L^2((\frac{1}{n},1))$
\begin{equation}
\label{Ltilderad}
\widetilde{{L^n}}\apice{m}_{\!\!\!\!p,rad}( v): = r^2\left(  -v''-\frac{(N-1)}{r}v'-p|u\apice{m}_p(r)|^{p-1}v\right), \ \ \ \ r\in (\frac{1}{n},1)
\end{equation}
and denote by \[\widetilde{\beta_i^n}(m,p), \quad i\in\N^+\] its eigenvalues  counted with their multiplicity. Obviously $\widetilde{\beta_i^n}(m,p)$ are nothing else than the radial eigenvalues of $\widetilde{{L^n}}\apice{m}_{\!\!\!\! p}$. Let us also set
\begin{eqnarray}\label{kappatilde}
&\widetilde{k_p^n}(m) := \#\{\mbox{negative eigenvalues $\widetilde{\mu_i^n}(m,p)$ of $\widetilde{{L^n}}\apice{m}_{\!\!\!\! p}$}\},
\\\label{kappatilderad}
&\widetilde{k^n}_{\!\!\!\!p,rad}(m) := \#\{\mbox{negative eigenvalues $\widetilde{\beta_i^n}(m,p)$ of the operator $
\widetilde{{L^n}}\apice{m}_{\!\!\!\!p,rad}$}\}.
\end{eqnarray}

Denoting by $\sigma(\cdot)$ the spectrum of a linear operator we recall that the following decomposition holds:
\begin{equation}  \label{lemma:decompositionOfTheSpectrum}
\sigma(\widetilde{{L^n}}\apice{m}_{\!\!\!\! p})
=\sigma(\widetilde{{L^n}}\apice{m}_{\!\!\!\!p,rad})+\sigma(-\Delta_{S^{N-1}}), \quad \mbox{for any}\;n\in\mathbb N^+,
\end{equation}
where $\Delta_{S^{N-1}}$ is the Laplace-Beltrami operator on the unit sphere $S^{N-1}$, $N\geq 3$. The proof of \eqref{lemma:decompositionOfTheSpectrum} is not difficult, it can be found for example in \cite{GladialiGrossiPacellaSrikanth}. So  \eqref{lemma:decompositionOfTheSpectrum} means that, for any $n\in\mathbb N^+$:
\begin{equation}\label{decomposizioneAutovalori} \widetilde{\mu_j^n}(m,p)\ =\ \widetilde{\beta_i^n}(m,p)\ +\ \lambda_k, \ \ \mbox{ for } i,j\in\N^+,\ \ k\in\N,
\end{equation}
where $\lambda_k$ are the eigenvalues of $-\Delta_{S^{N-1}}$, $N\geq 3$. Note that in \eqref{decomposizioneAutovalori} only $\widetilde{\beta_i^n}(m,p)$
depend on the exponent
$p$, while the eigenvalues $\lambda_k$
depend only on the dimension
$N$ and it is known (\cite[Proposition 4.1]{BerezinShubin}) that
\begin{equation}\label{lambdak}
\lambda_k=k(k+N-2),\ \ k\in\N,
\end{equation}
with multiplicity
\begin{equation}\label{multiplicity}
N_k-N_{k-2},
\end{equation}
where
\begin{equation}\label{N_h}
N_h:=\binom{N-1+h}{N-1}=\frac{(N-1+h)!}{(N-1)!h!},\ \mbox{ if }h\geq 0, \ \ \ N_h=0,\ \mbox{ if }h< 0.
\end{equation}

Next result shows the equivalence between the number of the negative eigenvalues of the linearized operator $L\apice{m}_p$ in $A_n$ and that of the weighted operators:
\begin{lemma} \label{lemmaEquivTraPesoESenzaPeso} We have: \[k^n_p(m)=\widetilde{k_p^n}(m)\qquad\mbox{ and }\qquad k^n_{p,rad}(m)=
\widetilde{k^n}_{\!\!\!\!p,rad}(m)
.\]
\end{lemma}
\begin{proof}
See \cite[Lemma 4.2]{DeMarchisIanniPacellaMathAnn}
\end{proof}

Combining Lemma \ref{lemma:morseProblemiSenzaPesoAnello}, Lemma \ref{lemmaEquivTraPesoESenzaPeso}, \eqref{LemmaMorseIndexRadiale} and \eqref{LemmaMorseIndexRadiale2} we get:

\begin{proposition}\label{proposition:MorseRadialeConPeso} Let $\in\N^+$ and $p\in (1,p_S)$. There exists $n'_p=n'_p(m)\in\N^+$ such that
\[\mathsf{m}(u\apice{\emph m}_p)=\widetilde{k_p^n}(m)\qquad\mbox{ and }\qquad \mathsf{m_{rad}}(u\apice{\emph m}_p)=\widetilde{k^n}_{\!\!\!\!p,rad}(m),\quad\mbox{  for }n\geq  n'_p.\]
 Hence
\begin{equation}
\label{corb}\widetilde{k_p^n}(m)\geq m+N(m-1)\qquad\mbox{ and }\qquad\widetilde{k^n}_{\!\!\!\!p,rad}(m)=m,\quad\mbox{ for }n\geq  n'_p.
\end{equation}
\end{proposition}

Because of the decomposition  \eqref{decomposizioneAutovalori} and of Proposition \ref{proposition:MorseRadialeConPeso} it is clear that in order to evaluate the Morse index $\mathsf{m}(u\apice{m}_p)$ (i.e. to prove Theorem \ref{teoPrincipaleMorse}) we have to estimate the negative eigenvalues $\widetilde{\beta_i^n}(m,p)$ of the weighted operator $\widetilde{{L^n}}\apice{m}_{\!\!\!\!p,rad}$ which, by \eqref{corb}, are only the first $m$ ones.

\

We conclude this section by an  estimate of the last negative eigenvalue $\widetilde{\beta_m^n}(m,p)$. This result generalizes to any $m\in\N^+$ the analogous one already proved in \cite[Proposition 4.5]{DeMarchisIanniPacellaMathAnn} in the case $m=2$.

\

We emphasize that an estimate of the other negative eigenvalues $\widetilde{\beta_i^n}(m,p)$, $i=1,\ldots, m-1$, is much more difficult and it will be the object of the next section.

\

\begin{proposition}\label{LemmaStimeAutovaloriRadialeConPeso Iparte}
Let $m\in\N^+$ and $p\in (1,p_S)$. Let $n''_p=n''_p(m):=[\frac1{r\apice{\emph m}_{1,p}}]+1$, where $r\apice{\emph m}_{1,p}$ is  the first nodal radius of $u\apice{\emph m}_p$ as defined in \eqref{rp}.
Then
\[\widetilde{\beta_m^n}(m,p)> -(N-1)\ \ \mbox{ for any }n\geq  n''_p.\]
\end{proposition}
\begin{proof}
 Let $\eta(r):=\displaystyle{\frac{\partial u\apice{m}_p(r)}{\partial r}}$, then by the choice of $n_p''$ it follows that for any $n\geq  n''_p$ one has $\frac{1}{n}<r\apice{m}_{1,p}$ and so the function $\eta$ satisfies
\[\left\{
\begin{array}{lr}
\widetilde{{L^n}}\apice{m}_{\!\!\!\!p,rad}\ \eta=-(N-1)\eta, \quad \ r\in (\frac{1}{n},1)
\\
\vspace{-8pt}
\\
\eta(\frac{1}{n})<0 \
\\
\vspace{-15pt}
\\
\eta(1)\lessgtr 0 \mbox{ for } m {\small{\begin{array}{ll}\mbox{odd}\\\vspace{-17pt}\\\mbox{even}\end{array}}}
\end{array}
\right.
\]
(the inequalities on the boundary deriving from the assumption  $u\apice{m}_p(0)>0$ in \eqref{MaxInZero}, moreover they are strict by the Hopf's Lemma).
Moreover we know that, for $n\geq  n''_p$,  $\eta$ has exactly $m-1$ zeros in the interval $(\frac 1n,1)$, given (if $m\geq 2$) by the points $s\apice{m}_{i,p}$, $i=1,\ldots, m-1$, defined in \eqref{sp}. \\
Let $w$ be an eigenfunction of $\widetilde{{L^n}}\apice{m}_{\!\!\!\!p,rad}$ associated with the eigenvalue $\widetilde{\beta_m^n}(m,p)$, namely
\[\left\{
\begin{array}{lr}
\widetilde{{L^n}}\apice{m}_{\!\!\!\!p,rad}\ w=\widetilde{\beta_m^n}(m,p)\, w, \quad \ r\in (\frac{1}{n},1)
\\
\vspace{-8pt}
\\
w(\frac{1}{n})=0
\\
\vspace{-10pt}
\\
w(1)= 0.
\end{array}
\right.
\]
It is well known that $w$ has exactly $m$ nodal regions.\\
Assume by contradiction that $\widetilde{\beta_m^n}(m,p)\leq-(N-1)$.\\
If $\widetilde{\beta_m^n}(m,p)=-(N-1)$, then $\eta$ and $w$ are two solutions of the same Sturm-Liouville equation
\[(r^{N-1}v')'+\left[p|u_p(r)|^{p-1}r^{N-1}+\frac{\widetilde{\beta_m^n}(m,p)}{r^{3-N}}\right]v=0,\quad \ \ \ r\in (\frac{1}{n},1)\]and they are linearly independent because $\eta(1)\neq 0=w(1)$.
As a consequence (Sturm Separation Theorem) the zeros of $\eta$ and $w$ must alternate. Since $\eta $ has $m-1$ zeros, $w$ must then have $m-1$ nodal regions and this gives a contradiction.
\\
 If $-(N-1)>\widetilde{\beta_m^n}(m,p)$, then by the Sturm Comparison Theorem, $\eta$ must have a zero between  any two consecutive zeros of $w$. As a consequence, since we know that $w$ has  $m-1$ zeros in $(\frac{1}{n},1)$ and that also the boundary points $\frac{1}{n}$ and $1$ are zeros, then $\eta$ must have $m$ zeros in $(\frac{1}{n},1)$,
which gives again a contradiction.
\end{proof}

\

\

\section{Asymptotic analysis of the eigenvalues $\widetilde{\beta_i^n}(m,p)$, $i=1,\ldots, m-1$ }\label{Section:Asymptotic4}

This section is devoted to study the asymptotic behavior, as $p\rightarrow p_S$, of  the first $(m-1)$ eigenvalues $\widetilde{\beta_i^n}(m,p)$, $i=1,\ldots,m-1$,  of the auxiliary weighted radial operator $\widetilde{{L^n}}\apice{m}_{\!\!\!\!p,rad}$ defined in \eqref{Ltilderad}, when $u\apice{m}_p$ is the radial solution to \eqref{problem} having $m$ nodal regions, for $m\in N^+$, which satisfies $u_p\apice{m}(0)>0$.

\

Recall that, for each $n\in \N^+$, the operator $\widetilde{{L^n}}\apice{m}_{\!\!\!\!p,rad}$ in \eqref{Ltilderad} is defined in the annulus
\[A_n=\{x\in\R^N\ :\ \frac{1}{n}<|x|<1\}.\]

For our purposes it is convenient to  \emph{chose the number $n$} in dependence of $p$ (and $m$) as follows:
\begin{equation}
\label{np}
n\apice{m}_p:=\max\{ n'_p, n''_p,[(M\apice{m}_{0,p})^{(p-1)}]+1\},
\end{equation}
where $n'_p=n'_p(m)$ is defined in Proposition \ref{proposition:MorseRadialeConPeso}, while $n''_p=n''_p(m)$ is as  in Proposition \ref{LemmaStimeAutovaloriRadialeConPeso Iparte}.
\\
Then for any $i\in\mathbb N^+$  we consider  the family of eigenvalues defined as
\begin{equation}\label{np1}
\widetilde\beta_i(m,p):=\widetilde{\beta_i^{n}}(m,p)\qquad \mbox{when }\ n=n\apice{m}_p.
\end{equation} 
Notice that the definition of $n\apice{m}_p$ in \eqref{np} and \eqref{corb}  imply that $\widetilde\beta_i(m,p)<0$, for $i=1,\ldots, m-1$, for every $p\in (1,p_S)$.\\ 
In order to shorten the notation for the operator, we  set:
\begin{equation}\label{defOpRadBreve}\widetilde L\apice{m}_{p,rad}:=\widetilde{{L^{n}}}\apice{m}_{\!\!\!\!p,rad}\quad \mbox{ when }n=n\apice{m}_p.
\end{equation}

\

The main result of this section is about the asymptotic behavior of the first eigenvalue $\widetilde\beta_1(m,p)$ as $p\rightarrow p_S$:

\begin{proposition}\label{theorem:limiteBeta1}
Let $m\in\N^+$.
\begin{equation}
\label{beta1maggioreUguale}
\liminf_{p\rightarrow p_S}\ \widetilde{\beta}_1(m,p)\geq -(N-1).
\end{equation}
\end{proposition}
An immediate consequence of  the previous proposition is the following:
\begin{corollary}\label{cor:limiteBetai}
Let $m\in\N^+$.
\[
\liminf_{p\rightarrow p_S}\ \widetilde{\beta}_i(m,p)\geq -(N-1), \qquad\mbox{for all }\ i=1,\ldots, m-1.
\]
\end{corollary}

\

\begin{remark} \label{remarkDopoProp4.2}
In the next section, while proving Theorem \ref{teoPrincipaleMorse}, we will show the reverse inequality: \[\widetilde{\beta_i}(m,p)<-(N-1),\qquad \forall\, i=1,\ldots, m-1, \quad \mbox{ for $p$ close to $p_S$}\] (see \eqref{stimadx}). Combining this with Corollary \ref{cor:limiteBetai} we will obtain the precise value of the limit:
\begin{equation} \label{limitePrecisodiBetap}
\widetilde{\beta}_i(m,p)\rightarrow -(N-1)\quad\mbox{  as }\ p\rightarrow p_S, \qquad\forall\, i=1,\ldots, m-1
\end{equation}
(see \eqref{convergenzadasinistraaltri}).
\end{remark}

\

The result in Proposition \ref{theorem:limiteBeta1} is the core of the proof of Theorem  \ref{teoPrincipaleMorse}. Since its proof is very long and needs various nontrivial estimates, let us first explain the strategy.

\

In order to get \eqref{beta1maggioreUguale} we consider, for any fixed $p\in (1, p_S)$,  the (radial and positive) eigenfunction  $\phi\apice{m}_p$  of $\widetilde L\apice{m}_{p,rad}$ (defined as in \eqref{defOpRadBreve}) associated with the first eigenvalue $\widetilde\beta_1(m,p)$, namely
\begin{equation}
\label{primaautofunz}
\left\{
\begin{array}{lr}
-{\phi\apice{m}_{p}}''-\frac{(N-1)}{r}{\phi\apice{m}_{p}}'-p|u\apice{m}_p|^{p-1}\phi\apice{m}_{p}=\widetilde{\beta}_1(m,p)\, \displaystyle{\frac{\phi\apice{m}_{p}}{r^2}}, \ \ \ \ r\in (\frac{1}{n\apice{m}_p},1)
\\
\phi\apice{m}_{p}(\frac{1}{n\apice{m}_p})=\phi\apice{m}_{p}(1)=0.
\end{array}
\right.
\end{equation}
To obtain the result one would like to pass to the limit as $p\rightarrow p_S$ into \eqref{primaautofunz} and deduce the value of $\lim_{p\rightarrow p_S}\widetilde{\beta}_1(m,p)$ by studying the limit eigenvalue problem. 

\

Since the term $p|u\apice{m}_p|^{p-1}$ in the equation \eqref{primaautofunz} is not bounded, it is more convenient to scale properly the eigenfunctions $\phi\apice{m}_p$ and pass to the limit into the equation satisfied by the scalings. The right possible scalings are the $\widehat{\phi\apice{m}_p}^i$, $i=0,\ldots, m-1$, defined in  \eqref{fiCappuccio} below, which satisfy the equations in \eqref{eq cappuccio p} where the eigenvalue $\widetilde{\beta}_1(m,p)$ again appears. Note that the scaling parameter in the definition of $\widehat{\phi\apice{m}_p}^i$ is given by the value $M\apice{m}_{i,p}$ of the  $L^{\infty}$-norm of $u\apice{m}_p$ in the corresponding  $i$-th nodal region. 

\

Of course this procedure is efficient if at least one among the $\widehat{\phi\apice{m}_p}^i$  does not vanish in the limit. Since we cannot guarantee that this is always the case (see \emph{CASE 2.} in the proof of Proposition \ref{theorem:limiteBeta1}) we combine it with  a different  
strategy which consists in considering a suitable \emph{limit eigenvalue problem}  (with the operator $\widetilde L^*$ in Section \ref{subs:limitEigenvalueProblem}) and exploiting the variational characterization  of its first eigenvalue. This reduces the proof to analyzing the \emph{difference} between  
a \emph{limit potential} $V$  and the actual potential $V\apice{m}_{0,p}$ defined in  \eqref{limiteV} and \eqref{defVpprima} below, exploiting the asymptotic behavior of $u\apice{m}_p$ studied in Section \ref{Section:Asymptotic2}.
In particular we need to   
 evaluate the contribution to the limit of $\widetilde\beta_1(m,p)$ given by the first nodal region $B\apice{m}_{0,p}$ of $u\apice{m}_p$, which is contained in  Lemma \ref{proposition:partePositiva} below, and  the contribution given by the other nodal regions of  $u\apice{m}_p$  and this is done in  Lemma \ref{proposition:StimaMaxASxDiEpsilon_p^-} and Lemma \ref{proposition:StimaMaxADxDiEpsilon_p^-} where the behavior of the function $f\apice{m}_p(r):=|r|^2|u\apice{m}_p(r)|^{p-1}$  in the  nodal regions $B\apice{m}_{i,p}$, $i=1,\ldots, m-1$, of $u\apice{m}_p$ is studied.

\

To make easier the understanding of the proof of Proposition \ref{theorem:limiteBeta1} we have divided this section as follows:
\begin{itemize}
\item in Section \ref{subs:limitEigenvalueProblem}  we introduce the limit weighted eigenvalue problem;
\item in Section \ref{subsecPrelim} we collect all the preliminary results about $\phi\apice{m}_p$ as well as the properties of its scalings $\widehat{\phi\apice{m}_p}^i$, $i=0,\ldots, m-1$;
\item in  Section \ref{subsectPositive} we estimate $u\apice{m}_p$ in $B\apice{m}_{0,p}$; 
\item in Section  \ref{subsectNegative} we estimate $u\apice{m}_p$ in $B\setminus B\apice{m}_{0,p}$;
\item in Section \ref{subsectionproofpro} we  complete the proof of Proposition \ref{theorem:limiteBeta1}.
\end{itemize}

 \

\subsection{A limit weighted eigenvalue problem}\label{subs:limitEigenvalueProblem} 

\

\

Let $N\geq 3$ and consider the weighted linear operator
\[\widetilde L^* v:=|x|^{2}\left[-\Delta v-V(x)v\right],\qquad x\in\mathbb R^N\]
where 
\begin{equation}
\label{limiteV}
V(x):= p_S\, U(x)^{p_S-1}=\frac{N+2}{N-2}\left( \frac{N(N-2)}{N(N-2)+|x|^2} \right)^{2}
\end{equation}
with $U$  as in \eqref{UNgeq3}, i.e. $U$ is the unique positive bounded solution to the critical equation
\eqref{criticalLimitEquation} in $\R^N$.

\

We want to define the first eigenvalue of $\widetilde L^*$.  
Let $D^{1,2}(\R^N)$ be the Hilbert space defined as  the  closure of $C^{\infty}_c(\R^N)$ with respect to the Dirichlet norm $\|v\|_{D^{1,2}(\R^N)}:=\left(\int_{\R^N}|\nabla v(x)|^2dx\right)^{\frac{1}{2}}$ and let us denote by $D^{1,2}_{rad}(\R^N)$ its subspace made of  radial functions.

\

Let us set
\begin{eqnarray}\label{defbetastar}
\widetilde\beta^*:=
\inf_{
\substack{
v\in D^{1,2}_{rad}(\R^N)\\ v\neq 0}}  \frac{\int_{\mathbb R^N}\left(|\nabla v(x)|^2-V(x)v(x)^2\right) dx}{\|\frac{v}{|x|}\|_{L^2(\R^N)}^2}.
\end{eqnarray}

Observe that this definition is well posed since  the Hardy inequality holds: 
\begin{equation}\label{Hardyinequa}
\|\frac{v}{|x|}\|_{L^2(\R^N)}\leq \frac{2}{(N-2)}\|v\|_{D^{1,2}(\R^N)},\qquad \mbox{for any $v\in D^{1,2}(\R^N), \quad N\geq 3$}
\end{equation} and so 
\[
\int_{\R^N}V(x)v(x)^2 dx\leq \sup_{\mathbb R^N}(V(x)|x|^2)\int_{\mathbb R^N}\frac{v(x)^2}{|x|^2}dx\overset{\eqref{Hardyinequa}}{\leq} C\|v\|^2_{D^{1,2}(\R^N)},
 \]
where we have used that $\sup_{\R^N}(V(x)|x|^2)<+\infty $.
\\
It is useful for the sequel  to introduce also the weighted  Hilbert space
\begin{equation}\label{def:L2Pesato}
L^2_{\frac{1}{|x|}}(\R^N):=\left\{v:\R^N\rightarrow \R \ : \ \frac{v}{|x|}\in L^2(\R^N)    \right\},
\end{equation}
endowed with the scalar product $(u,v):=\int_{\R^N}\frac{u(x)v(x)}{|x|^2}dx$. Note that
 $D^{1,2}_{rad}(\R^N)\hookrightarrow L^2_{\frac{1}{|x|}}(\R^N)$ continuously by Hardy inequality.\\
In \cite{DeMarchisIanniPacellaMathAnn} the precise value of $\widetilde\beta^*$ has been computed in any dimension  and this will be a crucial step towards the proof of Theorem \ref{teoPrincipaleMorse}. We summarize the results for $\widetilde\beta^*$ obtained in \cite{DeMarchisIanniPacellaMathAnn} in the next theorem.
\begin{theorem}\label{lemma:betastar} For any $N\geq 3$
\[\widetilde\beta^*=-(N-1) \]
and it is achieved at the function
\[
\eta^*(x)= \frac{|x|}{(1+\frac{|x|^{2}}{N(N-2)})^{\frac{N}{2}}}, 
\]
which solves the eigenvalue problem 
\begin{equation}\label{equazioneLimite}
 -\Delta\eta(x)-V(x)\eta(x)=\lambda\frac{\eta(x)}{|x|^2}\qquad x\in \R^N\setminus\{0\}\\
 \end{equation}
 with eigenvalue 
 \[\lambda =  \widetilde\beta^*.\]
Moreover if there exists  $\eta\in C^2( \R^N\setminus\{0\})\cap D^{1,2}_{rad}(\R^N)$, $\eta\geq0$, $\eta\neq 0$ radial solution to \eqref{equazioneLimite} with $\lambda\leq 0$, then \begin{equation}\label{unicheSoluzioniDiEquazioneLimite}\lambda=-(N-1),
 \end{equation}
namely $\widetilde\beta^*$ is the unique nonpositive radial eigenvalue for problem \eqref{equazioneLimite}.
\end{theorem}
\begin{proof}
See Section $5$ of \cite{DeMarchisIanniPacellaMathAnn}.
\end{proof}

\

\subsection{Properties of the eigenfuntion and its scalings}\label{subsecPrelim}

\

\

For any $m\in\N^+$ and $p\in (1,p_S)$ let us set
\begin{equation}
\label{definition:Ap}
A\apice{m}_p:=A_{n\apice{m}_p}=\big\{y\in \mathbb R^N\ :\ \frac{1}{n\apice{m}_p}<|y|<1 \big\}
\end{equation}
with $n\apice{m}_p$ defined in \eqref{np} and let $\phi\apice{m}_p$ be the (radial and positive) solution to \eqref{primaautofunz} normalized in such a way that
\begin{equation}\label{normalizzazione}
\left\|\frac{\phi\apice{m}_{p}}{|y|}\right\|_{L^2(A\apice{m}_p)}=1.
\end{equation}

\

\begin{lemma}\label{lemma:boundGradiente} For any $m\in\N^+,$ there exist $\delta=\delta(m)>0$ and $C>0$ (independent of $m$) such that
\[
\sup\{\|\nabla\phi\apice{\emph m}_p\|^2_{L^2(A\apice{m}_p)}\,:\,p\in(p_S-\delta,p_S)\}\leq C.
\]
\end{lemma}

\begin{proof}
From \eqref{primaautofunz} and recalling that, by  \eqref{Q3Nge3}, there exists $\delta=\delta(m)>0$ such that $p|u\apice{m}_p(y)|^{p-1}|y|^2\leq C$, for any $y\in B$ and  $p>p_S-\delta$, we have:
\begin{eqnarray}
\int_{A\apice{m}_p}|\nabla\phi\apice{m}_p(y)|^2dy 
&= &\int_{A\apice{m}_p}p|u\apice{m}_p(y)|^{p-1}|y|^2\frac{\phi\apice{m}_p(y)^2}{|y|^2} dy+\widetilde\beta_1(m,p)\int_{A\apice{m}_p}\frac{\phi\apice{m}_p(y)^2}{|y|^2}dy
\nonumber\\
&\leq & C\int_{A\apice{m}_p}\frac{\phi\apice{m}_p(y)^2}{|y|^2}dy +\widetilde\beta_1(m,p)\int_{A\apice{m}_p}\frac{\phi\apice{m}_p(y)^2}{|y|^2}dy
\nonumber\\
&\overset{\eqref{normalizzazione}}{=}& C +\widetilde\beta_1(m,p)\label{comeQui}\\
&\leq & C, \nonumber
\end{eqnarray}
since  $\widetilde\beta_1(m,p)<0$.
\end{proof}

\

Next result gives a first, still inaccurate, bound from below of  $\widetilde{\beta}_1(m,p)$ that will be useful in the sequel.

\begin{lemma}\label{lemma:beta1pbounded}
For any $m\in\N^+,$ there exist $\delta=\delta(m)>0$ and $C>0$ (independent of $m$) such that
\begin{equation}\label{betaUnoLimitato}
-C\leq \widetilde{\beta}_1(m,p)\ (< 0),\qquad\mbox{for any \ $p\in (p_S-\delta,p_S)$}.
\end{equation}
\end{lemma}

\begin{proof}
The proof follows directly from  \eqref{comeQui}.
\end{proof}

\

Let 
\begin{equation}\label{defApCapp}
\widehat{A\apice{m}_p}^i:=(M\apice{m}_{i,p})^{\frac{p-1}{2}}A\apice{m}_p= \Big\{y\in \mathbb R^N\ :\ \frac{(M\apice{m}_{i,p})^{\frac{p-1}{2}}}{n\apice{m}_p}<|y|<(M\apice{m}_{i,p})^{\frac{p-1}{2}} \Big\}, 
\end{equation}
for $i=0,\ldots, m-1$, where $A\apice{m}_p$ is as in \eqref{definition:Ap} and  consider the $m$  scalings of $\phi\apice{m}_p$, defined by
\begin{equation}\label{fiCappuccio}
\widehat{\phi\apice{m}_p}^i(x):=\frac{1}{(M\apice{m}_{i,p})^{\frac{(p-1)(N-2)}{4}}}\phi_{p}\apice{m}\Big(\frac{|x|}{(M\apice{m}_{i,p})^{\frac{p-1}{2}}}\Big),\quad \mbox{ for }x\in \widehat{A\apice{m}_p}^i, \quad i=0,\ldots, m-1,
\end{equation}
which, by \eqref{primaautofunz}, satisfy the equations  
\begin{equation}\label{eq cappuccio p}
\left\{
\begin{array}{lr}
-\Delta\widehat{\phi\apice{m}_p}^i-V\apice{m}_{i,p}(x)\widehat{\phi\apice{m}_p}^i=\widetilde{\beta}_1(m,p) \displaystyle{\frac{\widehat{\phi\apice{m}_p}^i}{|x|^2}}, \ \ \ \ x\in \widehat{A\apice{m}_p}^i
\\
\widehat{\phi\apice{m}_p}^i=0\ \ \ \mbox{ on }\ \partial \widehat{A\apice{m}_p}^i
\end{array}
\right.
\end{equation}
where
\begin{equation}\label{defVpprima}
V\apice{m}_{i,p}(x)
:=   p\frac{1}{(M\apice{m}_{i,p})^{p-1}}\ \Big|u\apice{m}_p\Big(\frac{|x|}{(M\apice{m}_{i,p})^{\frac{p-1}{2}}}\Big)\Big|^{p-1}.
\end{equation}
Note that by \eqref{epsilonpmN3}, \eqref{np} and  \eqref{limiterappepsilon} we have that
\begin{equation}
\widehat{A\apice{m}_p}^i\to\R^N\setminus\{0\} \qquad\mbox{ as $p\to p_S$, $\ \forall\ i=0,\ldots, m-1$.}
\end{equation}
Moreover observe that when $x\in \widetilde{T}\apice{m}_{i,p}\cap \widehat{A\apice{m}_p}^i$ 
\begin{equation}
\label{defVp}
V\apice{m}_{i,p}(x)=
p|z\apice{m}_{i,p}(x)|^{p-1},\ i=0,\ldots, m-1
\end{equation}
where $\widetilde{T}\apice{m}_{i,p}$ and $z\apice{m}_{i,p}$ are the rescaled sets and functions defined in \eqref{zeta}, hence by Theorem \ref{theorem:analisiAsintoticaCasoNgeq3}, we have that, as $p\to p_S$:
\begin{equation}\label{npepsilonp}
\widetilde{T}\apice{m}_{i,p}\cap \widehat{A\apice{m}_p}^i= \left\{\begin{array}{ll}
\widehat{A\apice{m}_p}^0 & \mbox{if }i=0\\
\widetilde{T}\apice{m}_{i,p} & \mbox{if }i=1,\ldots, m-1
\end{array} \right\}
\longrightarrow\R^N\setminus\{0\}, \qquad\mbox{ $\ \forall\ i=0,\ldots, m-1$}
\end{equation}
and also that
\begin{eqnarray}\label{V_ptoVN}
&& V\apice{m}_{0,p} \longrightarrow V
 \ \mbox{ in }\ C^0_{loc}(\mathbb R^N)\
\\
\label{W_ptoWN}
&&
V\apice{m}_{i,p}\chi_{\widetilde{T}\apice{m}_{i,p}} \longrightarrow V
 \ \mbox{ in }\ C^0_{loc}(\R^N\setminus\{0\}), \qquad \forall \ i=1,\ldots, m-1,
\end{eqnarray}
where $V$ is defined in \eqref{limiteV}.

\

Still denoting by $\widehat{\phi\apice{m}_p}^i$ the extension to $0$ of $\widehat{\phi\apice{m}_p}^i$ outside of $\widehat{A\apice{m}_p}^i$, we also have that $\widehat{\phi\apice{m}_p}^i$ is bounded in $D^{1,2}_{rad}(\R^N)$, indeed:

\begin{lemma}\label{lemma:boundrescalatefi}
For any $m\in\N^+,$ there exist $\delta=\delta(m)>0$ and $C>0$ (independent of $m$) such that
\begin{equation} \label{boundGradienteRiscalate}
\sup \{ \| \nabla \widehat{\phi\apice{\emph m}_p}^i\|_{L^2(\R^N)}\,:\,p\in(p_S-\delta,p_S)\}\leq C.
\end{equation}
Moreover
\begin{equation}\label{normaPesata1}
\left\| \frac{\widehat{\phi\apice{\emph m}_p}^i}{|x|}   \right\|_{L^2(\R^N)}=1.
\end{equation}
\end{lemma}

\

\begin{proof}
The proof of \eqref{boundGradienteRiscalate} and \eqref{normaPesata1} follows directly from  the definitions of
$\widehat{\phi\apice{m}_p}^i$.  Indeed we have
\[\int_{\R^N} \frac{\widehat{\phi\apice{m}_p}^i(x)^2}{|x|^2}dx 
=\int_{A\apice{m}_p} \frac{\phi\apice{m}_{p}(y)^2}{|y|^2} dy\overset{\eqref{normalizzazione}}{=}1\]
and, observing that
$
\nabla \widehat{\phi\apice{m}_p}^i(x) =(M\apice{m}_{i,p})^{-\frac{N(p-1)}{4}} \nabla\phi\apice{m}_p\Big(\frac{|x|}{(M\apice{m}_{i,p})^{\frac{p-1}{2}}}\Big)
$, we also get
\begin{equation}\label{bbbb}
\int_{\R^N} |\nabla \widehat{\phi\apice{m}_p}^i(y)|^2 dy 
=\int_{A\apice{m}_p} |\nabla\phi\apice{m}_p(x)|^2 dx\leq C
\end{equation}
by Lemma \ref{lemma:boundGradiente}.
\end{proof}

\

\subsection{An estimate in the first nodal region }
\label{subsectPositive}

\

\

In this section, investigating accurately the contribution given by the restriction of $u\apice{m}_p$ to  the \emph{first} nodal region $B\apice{m}_{0,p}$ intersected with the annulus $A\apice{m}_p$ introduced in \eqref{definition:Ap}, we derive an estimate that will be used later  in the proof of Proposition \ref{theorem:limiteBeta1}.\\
More precisely we consider the set
\begin{equation}
\label{definition:App} F\apice{m}_p:=A\apice{m}_p\cap B\apice{m}_{0,p}=\big\{y\in \mathbb R^N\ :\ \frac{1}{n\apice{m}_p}<|y|< r\apice{m}_{1,p} \big\}\overset{\eqref{np}}{\not = }\emptyset,
\end{equation}
 where  $n\apice{m}_p$ is defined in \eqref{np} and $r\apice{m}_{1,p}$ is the first nodal radius of $u\apice{m}_p$ (see \eqref{rp}) and prove the following:

\begin{lemma}\label{proposition:partePositiva}
Let $m\in\N^+$. For any $\varepsilon>0$ there exists $R_{\varepsilon}>0$ (independent of $m$) such that
\[\lim_{p\rightarrow p_S}\int_{\widehat{F\apice{\emph m}_p}^0\cap\{|x|> R\}} V\apice{\emph m}_{0,p}(x) \, \widehat{\phi\apice{\emph m}_p}^0(x)^2\, dx\ \leq\  \varepsilon,\qquad \mbox{ for all } R\geq R_{\varepsilon},\]
where   
\begin{equation}\label{def:Pprisclata}
\widehat{F\apice{\emph m}_p}^0:=(M\apice{\emph m}_{0,p})^{\frac{p-1}{2}}F\apice{\emph m}_p,
\end{equation}
$\widehat{\phi\apice{\emph m}_p}^0$ is as in \eqref{fiCappuccio} and $V\apice{\emph m}_{0,p}$ satisfies  \eqref{defVp}.
\end{lemma}

\begin{proof}
We divide the proof into two steps.

\

\emph{STEP 1.
We show that for any $R>0$
\begin{equation}
\lim_{p\rightarrow p_S}
\int_{\widehat{F\apice{m}_p}^0\cap\{|x|> R\}}
|z\apice{m}_{0,p}(x)|^{\frac{N}{2}(p-1)} dx\ =\
\int_{\{|x|>R\}} U(x)^{\frac{2N}{N-2}}dx,
\end{equation}
where $U$ is the function in \eqref{UNgeq3}.}

\

\emph{Proof of STEP 1.} On one side by the choice of $n\apice{m}_p$ in \eqref{np} we have that
\begin{equation}\label{preced}
\int_{\{|y|<\frac{1}{n\apice{m}_p}\}}|u\apice{m}_p(y)|^{\frac{N}{2}(p-1)} dy\leq \omega_N\frac{(M\apice{m}_{0,p})^{\frac{N}{2}(p-1)}}{(n\apice{m}_p)^N}\overset{\eqref{np}}{\leq}\frac{1}{(n\apice{m}_p)^{\frac{N}{2}}}\underset{p\rightarrow p_S}{\longrightarrow} 0,
\end{equation}
so, by the definition of $z\apice{m}_{0,p}$ (see \eqref{zeta}), by \eqref{limiteMezzaNormap} and \eqref{preced} we have
\begin{eqnarray}\label{oneside}
\int_{\widehat{F\apice{m}_p}^0}
|z\apice{m}_{0,p}(x)|^{\frac{N}{2}(p-1)} dx\ &=& \ \int_{F\apice{m}_p}|u\apice{m}_p(y)|^{\frac{N}{2}(p-1)} dy\nonumber\\
& = & \int_{B\apice{m}_{0,p}}|u\apice{m}_p(y)|^{\frac{N}{2}(p-1)} dy-\int_{\{|y|<\frac{1}{n\apice{m}_p}\}}|u\apice{m}_p(y)|^{\frac{N}{2}(p-1)} dy
\nonumber
\\
& \overset{\eqref{limiteMezzaNormap}+\eqref{preced}}{\underset{p\rightarrow p_S} {\longrightarrow}} &  S_N^{\frac{N}{2}}\stackrel{\eqref{normaU}}{=}\int_{\mathbb R^N} U(x)^{\frac{2N}{N-2}}dx.
\end{eqnarray}
On the other side  as $p\rightarrow p_S$, since   $z\apice{m}_{0,p}\rightarrow U$ in $C^2_{loc}(\mathbb R^N)$, $r\apice{m}_{1,p} (M\apice{m}_{0,p})^{\frac{p-1}{2}}\rightarrow +\infty$   by $(\mathcal{A}\apice{m}_{1})$ (which holds by Proposition \ref{proposition:AmiValePerOgnim}) and  \eqref{preced} holds, we deduce
\begin{equation}\label{otherside}
\int_{\widehat{F\apice{m}_p}^0\cap \{|x|\leq R\}}
|z\apice{m}_{0,p}(x)|^{\frac{N}{2}(p-1)} dx \underset{p\rightarrow p_S} {\longrightarrow} \int_{\{|x|\leq R\}} U(x)^{\frac{2N}{N-2}}dx,
\end{equation}
for any $R>0$. Combining \eqref{oneside} and \eqref{otherside} we get
\[
\int_{\widehat{F\apice{m}_p}^0\cap \{|x|> R\}}
|z\apice{m}_{0,p}(x)|^{\frac{N}{2}(p-1)} dx \  
\underset{p\rightarrow p_S} {\longrightarrow}\int_{\{|x|> R\}} U(x)^{\frac{2N}{N-2}}dx
\]

\

\emph{STEP 2. End of the proof.}

\

\emph{Proof of STEP 2.} By using H\"older inequality with exponents $\frac{N}{2}$, $\frac{N}{N-2}$, the Sobolev embedding theorem and Lemma \ref{lemma:boundrescalatefi}
 we get, for any $R>0$ and for any $p> p_S-\delta$ (where $\delta=\delta(m)$  as in Lemma \ref{lemma:boundrescalatefi}):
\begin{eqnarray}\label{holderineqdis}
&&\hspace{-1.3cm}\int_{\widehat{F\apice{m}_p}^0\cap\{|x|> R\}}
V\apice{m}_{0,p}(x)\, \widehat{\phi\apice{m}_p}^0(x)^2\, dx\ =\nonumber \\
&&\qquad\overset{\eqref{defVp}}{=}
\int_{\widehat{F\apice{m}_p}^0\cap\{|x|> R\}}
p|z\apice{m}_{0,p}(x)|^{p-1}\, \widehat{\phi\apice{m}_p}^0(x)^2\, dx \nonumber
\\
&&\qquad\overset{\mbox{\scriptsize{ H\"older}}}{\leq} 
p_S \left[\int_{\widehat{F\apice{m}_p}^0\cap\{|x|> R\}} |z\apice{m}_{0,p}(x)|^{\frac{N}{2}(p-1)}dx \right]^{\frac{2}{N}}
\ \left\|\widehat{\phi\apice{m}_p}^0\right\|^2_{L^\frac{2N}{N-2}(\mathbb R^N)}
\nonumber
\\
&&\qquad\overset{\mbox{\scriptsize{ Sobolev}}} {\leq }
\frac{p_S}{\sqrt{S_N}} \left[\int_{\widehat{F\apice{m}_p}^0\cap\{|x|> R\}} |z\apice{m}_{0,p}(x)|^{\frac{N}{2}(p-1)}dx \right]^{\frac{2}{N}}
\ \left\|\nabla\widehat{\phi\apice{m}_p}^0\right\|^2_{L^2(\mathbb R^N)}
\nonumber
\\
&&\qquad\overset{\mbox{\scriptsize{ Lemma \ref{lemma:boundrescalatefi}}}}{\leq}  C \left[\int_{\widehat{F\apice{m}_p}^0\cap\{|x|> R\}} |z\apice{m}_{0,p}(x)|^{\frac{N}{2}(p-1)}dx \right]^{\frac{2}{N}}.
\end{eqnarray}
Let $\varepsilon>0$ and  $R_{\varepsilon}>0$ such that
\begin{equation}\label{RgrandeU}
\int_{\{|x|>R\}} U(x)^{\frac{2N}{N-2}}dx\leq \frac{\varepsilon}{C} \ \mbox{ for }\ R\geq R_{\varepsilon}.
\end{equation}
Passing to the limit into \eqref{holderineqdis}, by {\sl STEP 1} and \eqref{RgrandeU} we then have
\begin{eqnarray*}
\lim_{p\rightarrow p_S}\int_{\widehat{F\apice{m}_p}^0\cap\{|x|> R\}}
V\apice{m}_{0,p}(x)\, \widehat{\phi\apice{m}_p}^0(x)\, ^2 dx
&\leq & \varepsilon \ \mbox{ for }\ R\geq R_{\varepsilon}.
\end{eqnarray*}
\end{proof}

\

\subsection{Estimates in the remaining nodal regions} \label{subsectNegative}

\

\

Let us consider the radial function $f\apice{m}_p$ defined in \eqref{Q3Nge3}:
\begin{equation} \label{definition:f_p}
f\apice{m}_{p}(y)=|y|^2 |u\apice{m}_p(y)|^{p-1}, \qquad y\in B. 
\end{equation}
The next two lemmas provide estimates of $f\apice{m}_p$ when $|y|$ belongs to suitable subsets of  $[r\apice{m}_{1,p},1]$,
where $r\apice{m}_{1,p}$ is the \emph{first nodal radius} of $u\apice{m}_p$ as defined in \eqref{rp}.

\

\begin{lemma}\label{proposition:StimaMaxASxDiEpsilon_p^-}
Let $m\in\N^+$. For any $\varepsilon>0$ there exists $\widehat{K_{\varepsilon}}(=\widehat{K_{\varepsilon}}(m))>1$ such that for any $K\geq \widehat{K_{\varepsilon}}$, there exists $\delta_{K,\varepsilon}(=\delta_{K,\varepsilon}(m))>0$ such that, for any $i=1,\ldots, m-1$, the set
\begin{equation}
\label{defG}
\emptyset\neq G\apice{\emph m}_{i,p, K} := \big\{y\in\R^N\ :\ r\apice{\emph m}_{i,p}< |y|< \frac{1}{K}(M\apice{\emph m}_{i,p})^{-\frac{p-1}{2}}\big\} \subset B\apice{\emph m}_{i,p}, \quad \mbox{ for }\  p\geq p_S-\delta_{K,\varepsilon}
\end{equation}
and
\begin{equation}\label{stimafmG} \max_{y\in \bigcup_{i=1}^{m-1}\overline{G\apice{\emph m}_{i,p, K}}}f\apice{\emph m}_p(y) \leq \varepsilon,\quad \mbox{ for }\  p\geq p_S-\delta_{K,\varepsilon}.
\end{equation}
%
\end{lemma}

\

\begin{proof}
Let us fix $i\in\{1\ldots, m-1\}$. 
Observe that by the limit properties \eqref{Cmi} and either $(\mathcal A\apice{m}_{i+1})$ when $i\neq m-1$ or \eqref{epsilonpmN3} when $i=m-1$ (see Proposition \ref{proposition:AmiValePerOgnim},  \ref{prop:Bm-1ValePerOgnim} and \ref{prop:BmiValePerOgnim} in  Section \ref{Section:Asymptotic2}), we get
\[
r\apice{m}_{i,p} (M\apice{m}_{i,p})^{\frac{p-1}{2}}\rightarrow 0\ \mbox{ and }\ r\apice{m}_{i+1,p} (M\apice{m}_{i,p})^{\frac{p-1}{2}}\rightarrow +\infty, \ \ \mbox{ as }p\rightarrow p_S.
\] 
So for any fixed $K>1$ there exists $\delta_{K,i}(=\delta_{K,i}(m))>0$ such that 
\begin{equation}\label{ordine1}
r\apice{m}_{i,p}<\frac{1}{K} (M\apice{m}_{i,p})^{-\frac{p-1}{2}}<r\apice{m}_{i+1,p},
\quad\mbox{ for }p\geq p_S-\delta_{K,i}.
\end{equation}
So for $K>1$ and  $p\geq p_S-\delta_{K,i}$ it is well defined
\[c_{K,p,i}(=c_{K,p,i}(m)):= \max_{y\in \overline{G\apice{\emph m}_{i,p, K}}}f\apice{m}_p(y). \]

Next we show that for any $\varepsilon>0$ there exists  $\widehat{K_{\varepsilon,i}}(=\widehat{K_{\varepsilon,i}}(m))>1$ such that for any
 $K\geq \widehat{K_{\varepsilon,i}}$, there exists $\delta_{K,i,\varepsilon}(=\delta_{K,i,\varepsilon}(m))\in (0,\delta_{K,i}]$ such that
\begin{equation}
\label{tesickpi}
 c_{K,p,i} \leq \varepsilon,\quad \mbox{ for }\ p\geq p_S-\delta_{K,i,\varepsilon}.
\end{equation}
Arguing by contradiction, we can assume that there exists $\alpha>0$ such that for all $n\in\N$, there exist $K_n(=K_n(m))\geq n$ and $p_n(=p_n(m))\geq p_S-\delta_{K_n,i}$ such that
\begin{equation}
c_{n,i}:=c_{K_n,p_{n},i}\geq \alpha^2 \label{assurdoC}.
\end{equation}
%
%
%
%
%
%
Since $p_n\geq p_S-\delta_{K_n,i}$, by \eqref{ordine1} we have that $r\apice{m}_{i,p_n}< \frac{1}{K_n} (M\apice{m}_{i,p_n})^{-\frac{p_n-1}{2}} <r\apice{m}_{i+1,p_n}$. For any $n\in\mathbb N$ let $r_n(=r_n(i,m))\in\R$  be the radius such that \[\left\{\begin{array}{lr} r\apice{m}_{i,p_n}\leq r_n\leq \frac{1}{K_n} (M\apice{m}_{i,p_n})^{-\frac{p_n-1}{2}}\\
\\
f\apice{m}_{p_n}(r_n)=(r_n)^2|u\apice{m}_{p_n}(r_n)|^{p_n-1}= c_{n,i}.\end{array}\right.\] Then 
\[(r_n)^2 (M\apice{m}_{i,p_n})^{p_n-1}= (r_n)^2 |u\apice{m}_{p_n}(s\apice{m}_{i,p_n})|^{p_n-1}\geq   (r_n)^2 |u\apice{m}_{p_n}(r_n)|^{p_n-1}= c_{n,i}\stackrel{\eqref{assurdoC}}{\geq} \alpha^2 >0.\]
On the other side by construction 
\[(r_n)^2(M\apice{m}_{i,p_n})^{p_n-1} \leq\frac{1}{(K_n)^2} \leq \frac{1}{n^2}, \quad \mbox{for all \ $n\in\N$}\]
which gives a contradiction and so proves \eqref{tesickpi}.\\
The conclusion of the proof follows  setting
\begin{eqnarray*}
&&\widehat{K_{\varepsilon}}(m) := \max\{\widehat{K_{\varepsilon,i}}(m),\  i=1,\ldots, m-1 \}\\
&& \delta_{K,\varepsilon}(m):=\min \{ \delta_{K,i,\varepsilon}(m),\ i=1,\ldots, m-1 \}
\end{eqnarray*}
so by \eqref{ordine1} we get \eqref{defG}, while  \eqref{tesickpi} proves \eqref{stimafmG}.
\end{proof}

\

\begin{lemma}\label{proposition:StimaMaxADxDiEpsilon_p^-}
Let $m\in\N^+$. For any $\varepsilon>0$ there exist $\delta_{\varepsilon}(=\delta_{\varepsilon}(m))>0$ and $K_{\varepsilon}(=K_{\varepsilon}(m))\geq \widehat{K_{\varepsilon}}$ (where $\widehat{K_{\varepsilon}}>1$ is defined in Lemma \ref{proposition:StimaMaxASxDiEpsilon_p^-}) such that for any $i=1,\ldots, m-1$ the set
\begin{equation}
\label{defH}
\emptyset\neq H\apice{\emph m}_{i,p, \varepsilon} :=\big\{y\in\R^N\ :\  K_{\varepsilon} (M\apice{\emph m}_{i,p})^{-\frac{p-1}{2}} < |y|< r\apice{\emph m}_{i+1,p}\big\}\subset B\apice{\emph m}_{i,p},\quad \mbox{ for }p\geq p_S- \delta_{\varepsilon}
\end{equation}
and
\begin{equation}\label{stimafmH} 
 \max_{y\in \bigcup_{i=1}^{m-1}\overline{H\apice{\emph m}_{i,p, \varepsilon}}}f\apice{\emph m}_p(y)\leq \varepsilon,\qquad \mbox{ for }p\geq p_S-\delta_{\varepsilon}.
 \end{equation}
\end{lemma}

\

\begin{proof}
We divide the proof into three steps.

\

\emph{STEP 1. Let $m\in\N^+$, $i\in\{1,\ldots, m-1\}$ and define \[ g\apice{m}_{p,i}(r):=\frac{ (M\apice{\emph m}_{i,p})^{p-1}r^2 }{\left[ 1+ \frac{2\alpha}{N (N-2)^2 } (M\apice{\emph m}_{i,p})^{p-1}r^2 \right]^{\frac{N-2}{2}(p-1)}},\] where $\alpha\in (0, \frac{N-2}{2})$ is fixed. We show that there exists $\widehat{K}>0$ (independent of $i$ and $m$) and $\widehat{\delta_i}(=\widehat{\delta_i}(m))>0$ such that: 
\[\widehat{K}(M\apice{\emph m}_{i,p})^{-\frac{p-1}{2}}<r\apice{\emph m}_{i+1,p},\ \mbox{ if }p\geq p_S-\widehat{\delta_i}
\]
and the function $g\apice{m}_{p,i}$ is monotone decreasing in  $[\widehat K(M\apice{\emph m}_{i,p})^{-\frac{p-1}{2}}, r\apice{\emph m}_{i+1,p}]$, for any $p\geq p_S-\widehat{\delta_i}$.}

\

\

\emph{Proof of STEP 1.} Let $\widehat{K}:=2\left[\frac{N (N-2)^2}{2\alpha  }\right]^{\frac{1}{2}}(>0)$. Since,  by \eqref{epsilonpmN3} for $i=m-1$ and property $(\mathcal A\apice{m}_{i+1})$  (which holds true  by Proposition \ref{proposition:AmiValePerOgnim}) for $i\neq m-1$, we have that 
\[r\apice{m}_{i+1,p}( M\apice{m}_{i,p})^{\frac{p-1}{2}}\rightarrow +\infty\ \mbox{ as }\ p\rightarrow p_S,\]
then there exists $\delta_{\widehat{K},i}(=\delta_{\widehat{K},i}(m))>0$ such that 
\[\widehat{K}(M\apice{m}_{i,p})^{-\frac{p-1}{2}}<r\apice{m}_{i+1,p},\ \mbox{ if }p\geq p_S-\delta_{\widehat{K},i}.
\] 
Moreover by easy computations
\begin{eqnarray*}
(g\apice{m}_{p,i})'(r)= \frac{2(M\apice{m}_{i,p})^{p-1}r}{\left[ 1+ \frac{2\alpha}{N (N-2)^2 } (M\apice{m}_{i,p})^{p-1}r^2 \right]^{\frac{(N-2)}{2}(p-1)+1}} \left[     1- \frac{\left[(p-1)(N-2)-2 \right]\alpha  (M\apice{m}_{i,p})^{p-1}}{N (N-2)^2 }r^2  \right]
\end{eqnarray*}
hence $(g\apice{m}_{p,i})'(r)\leq 0$ if and only if 
\[r\geq \left[\frac{N (N-2)^2 }{\left[(p-1)(N-2)-2 \right]\alpha}\right]^{\frac{1}{2}}(M\apice{m}_{i,p})^{-\frac{p-1}{2}}
%
%
.\]
Since by our choice of $\widehat K$ we have
\[\left[\frac{N (N-2)^2 }{\left[(p-1)(N-2)-2 \right]\alpha  }\right]^{\frac{1}{2}}\longrightarrow \frac{\widehat K }{2  }\ \mbox{ as }\ p\rightarrow p_S,\] there exists $\widetilde{\delta}>0$ such that
if $p>p_S-\widetilde{\delta}$ then  $(g\apice{m}_{p,i})'(r)\leq 0$ for $r\geq \widehat{K}(M\apice{m}_{i,p})^{-\frac{p-1}{2}}$. To conclude the proof of \emph{STEP 1} it is enough to take $\widehat{\delta_i}(=\widehat{\delta_i}(m)):=\min\{\delta_{\widehat{K},i}(m),\widetilde{\delta}\}$.

\

\

\emph{STEP 2. Let $m\in\N^+$. Let us fix $i\in\{1,\ldots, m-1\}$ and $\varepsilon>0$. We show that there exist $\delta_{\varepsilon,i}(=\delta_{\varepsilon,i}(m))>0$ and $K_{\varepsilon}(=K_{\varepsilon}(m))\geq \widehat{K_{\varepsilon}}$ (where $\widehat{K_{\varepsilon}}>1$ is defined in Lemma \ref{proposition:StimaMaxASxDiEpsilon_p^-}) such that 
\begin{equation}
\label{Hstar}
\emptyset\neq H\apice{\emph m}_{i,p, \varepsilon} \subset B\apice{\emph m}_{i,p},\quad \mbox{ for }p\geq p_S- \delta_{\varepsilon,i}
\end{equation}
and
\begin{equation}\label{stimafmHstar} 
 \max_{y\in \overline{H\apice{\emph m}_{i,p, \varepsilon}}}f\apice{\emph m}_p(y)\leq \varepsilon,\qquad \mbox{ for }p\geq p_S-\delta_{\varepsilon,i}.
 \end{equation}
}

\

\emph{Proof of STEP 2.} By  Proposition \ref{StimaFondamentaleNegativa} and Corollary \ref{cor:RmiSatisfied} we know that there exist $\gamma=\gamma(\alpha, m)\in (0,1)$, $\gamma(\alpha,m)\rightarrow 1$ as $\alpha\rightarrow 0$ and $\delta_i=\delta_i(\alpha,m)>0$ such that 
\begin{equation}\label{bound parte negativaFONDAMENTALE2}
f\apice{m}_p(r)\leq g\apice{m}_{p,i}(r), \qquad \mbox{ for }\ r\in (\gamma^{-\frac{1}{N}}s\apice{m}_{i,p}, r\apice{m}_{i+1,p}],\ p\geq p_S-\delta_i
\end{equation}
Observe that by property \eqref{Bmi} (which holds true by Propositions \ref{prop:Bm-1ValePerOgnim} and \ref{prop:BmiValePerOgnim}) 
\[s\apice{m}_{i,p}(M\apice{m}_{i,p})^{\frac{p-1}{2}}\rightarrow 0,\quad \mbox{ as }\ p\rightarrow p_S,\] so there exists $\widetilde\delta_{\widehat K,i}(m)>0$ such that
\begin{equation}\label{intervalli}\gamma^{-\frac{1}{N}}s\apice{m}_{i,p}<\widehat K (M\apice{m}_{i,p})^{-\frac{p-1}{2}}, \ \ \mbox{ for }  p\geq p_S-\widetilde\delta_{\widehat K,i}(m)
\end{equation}
where $\widehat K$ is the number obtained  in \emph{STEP 1.}
\\
Observe also that since,  by \eqref{epsilonpmN3} for $i=m-1$ and property $(\mathcal A\apice{m}_{i+1})$  (which holds true  by Proposition \ref{proposition:AmiValePerOgnim}) for $i\neq m-1$, we have that 
\[r\apice{m}_{i+1,p}( M\apice{m}_{i,p})^{\frac{p-1}{2}}\rightarrow +\infty\ \mbox{ as }\ p\rightarrow p_S,\]
then for any $K\geq\widehat{K}$ there exists $\delta_{K,i}(m)>0$ such that
\begin{equation}\label{intervalli2}\widehat K (M\apice{m}_{i,p})^{-\frac{p-1}{2}}\leq K (M\apice{m}_{i,p})^{-\frac{p-1}{2}} < r\apice{m}_{i+1,p} , \ \ \mbox{ for }  p\geq p_S-\delta_{K,i}(m).
\end{equation}
By \emph{STEP 1}, \eqref{bound parte negativaFONDAMENTALE2}, \eqref{intervalli} and \eqref{intervalli2} we have that for any $K\geq \widehat K$ and for $p\geq p_S-\min \{\delta_i(\alpha,m),\widetilde\delta_{\widehat K,i}(m), \delta_{K,i}(m), \widehat{\delta_i}\}$ (where $\widehat{\delta_i}(=\widehat{\delta_i}(m))$ is the one in \emph{STEP 1})
\begin{equation}\label{fPiccola}
f\apice{m}_p(r)\leq g\apice{m}_{p,i}(r)\leq g\apice{m}_{p,i}(K(M\apice{m}_{i,p})^{-\frac{p-1}{2}})\qquad\mbox{ for } r\in ( K(M\apice{m}_{i,p})^{-\frac{p-1}{2}},r\apice{m}_{i+1,p}].
\end{equation}

Moreover  if  $p> p_S-\frac{2}{N-2}$ then $\frac{N-2}{2}(p-1)>1$,  so
\begin{equation}\label{gvaazero}
g\apice{m}_{p,i}(K(M\apice{m}_{i,p})^{-\frac{p-1}{2}})=\frac{ K^2 }{\left[ 1+ \frac{2\alpha}{N (N-2)^2 } K^2 \right]^{\frac{N-2}{2}(p-1)}}\longrightarrow 0\qquad \mbox{ as }\ K\rightarrow +\infty
\end{equation}
The conclusion follows combining \eqref{gvaazero} with \eqref{fPiccola}.

\

\

\emph{STEP 3. Conclusion.}

\

\emph{Proof of STEP 3.} The proof follows by \emph{STEP 2.} taking
\[\delta_{\varepsilon}(=\delta_{\varepsilon}(m)):=\min\{\delta_{\varepsilon,i}(m), \ i=1,\ldots, m-1  \}.\]
\end{proof}

\

\subsection{Proof of Proposition \ref{theorem:limiteBeta1}}
\label{subsectionproofpro}

\

\

\begin{proof}[Proof] 
Arguing by contradiction let us assume that \eqref{beta1maggioreUguale} does not hold. Then  there exist $\varepsilon >0$ and a sequence $p_j\rightarrow p_S$,  as $j\rightarrow +\infty$, such that
\begin{equation} \label{stimasx}
\widetilde{\beta}_1(m,p_j)\rightarrow -(N-1)-10\,\varepsilon,\ \ \mbox{ as}\ j\rightarrow +\infty.
\end{equation}
Corresponding to this number $\varepsilon>0$ we can take $K_{\varepsilon}>1$  as in Lemma \ref{proposition:StimaMaxADxDiEpsilon_p^-}. Then we consider the $m-1$ scalings $\widehat{\phi\apice{m}_{p_j}}^i$, $i=1,\ldots, m-1$, defined  in \eqref{fiCappuccio} and observe that, by \eqref{normaPesata1} 
\[\liminf_{j\rightarrow +\infty} \int_{\{|x|\in[\frac{1}{K_{\varepsilon}},K_{\varepsilon}]\}}\frac{\widehat{\phi\apice{m}_{p_j}}^i(x)^2}{|x|^2}dx\quad \in [0,1],\qquad \forall\, i=1,\ldots, m-1.\]
Hence there exists a subsequence, that we still denote by $p_j$, for which  one of the following two statements holds:\\\\
{\emph{CASE 1.}}
There exists $\alpha_{\varepsilon}\in(0,1]$ and $\kappa\in\{1, \ldots, m-1\} $ such that:
\begin{equation}
\label{case1}
 \int_{\{|x|\in[\frac{1}{K_{\varepsilon}},K_{\varepsilon}]\}}\frac{\widehat{\phi\apice{m}_{p_j}}^{\kappa}(x)^2}{|x|^2}dx\geq \alpha_{\varepsilon},  \quad\forall\, j\in\N.
\end{equation}
{\emph{CASE 2.}} 
\begin{equation}\label{assumptionCASE2} \int_{\{|x|\in[\frac{1}{K_{\varepsilon}},K_{\varepsilon}]\}}\frac{\widehat{\phi\apice{m}_{p_j}}^i(x)^2}{|x|^2}dx\longrightarrow 0\quad\mbox{ as }\ j\rightarrow +\infty, \qquad\forall\, i=1,\ldots, m-1.
\end{equation}

\

In \emph{CASE 1}  we will prove that 
\begin{equation}\label{conclusion:case1}
\widetilde\beta_1(m,p_{j})\rightarrow -(N-1),\quad \mbox{ as }\ j\rightarrow +\infty,
\end{equation}
which contradicts \eqref{stimasx}.
\\
In \emph{CASE 2} we will show that there exists $j_{\varepsilon}\in\N$ such that 
\begin{equation}
\label{conclusion:case2}
\widetilde{\beta}_1(m,p_{j})\geq  -(N-1) -9\,\varepsilon, \quad \mbox{for any $j\geq j_{\varepsilon}$,}
\end{equation}
which also contradicts \eqref{stimasx}. So the assertion \eqref{beta1maggioreUguale} will be proved.

\

\

{\emph{Proof in  CASE 1.}}

\

We will pass to the limit as $j\rightarrow + \infty$ into the equation \eqref{eq cappuccio p} satisfied by  the scaling $\widehat{\phi\apice{m}_{p_j}}^{\kappa}$. Since \eqref{npepsilonp} implies that, for any fixed $\rho\in C^{\infty}_0(\R^N\setminus\{0\})$,  $\supp (\rho)\subset (\widetilde{T}\apice{m}_{\kappa,p_j}\cap\widehat{A\apice{m}_{p_j}}^{\kappa})$ for $j$ sufficiently large, by \eqref{eq cappuccio p} we have
\begin{equation}\label{equaziocappuccio}
\int_{\R^N \setminus\{0\}}\!\!\nabla \widehat{\phi\apice{m}_{p_j}}^{\kappa}\nabla\rho\ dx
\
-
\
\int_{\R^N \setminus\{0\}}\!\! V\apice{m}_{\kappa,p_j}(x)\widehat{\phi\apice{m}_{p_j}}^{\kappa}\rho\ dx\
-
\
\widetilde{\beta_1}(m,p_j)\int_{\R^N \setminus\{0\}}\!\!\frac{\widehat{\phi\apice{m}_{p_j}}^{\kappa}\rho}{|x|^2}\ dx=0,
\end{equation}
where in particular $V\apice{m}_{\kappa,p_j}$ satisfies \eqref{defVp}.\\
By Lemma \ref{lemma:boundrescalatefi} we know that $\widehat{\phi\apice{m}_{p_j}}^{\kappa}$ is bounded 
in the reflexive space $D^{1,2}_{rad}(\R^N)$, hence there exists $\widehat{\phi}=\widehat{\phi}\apice{m}_{\kappa} \in D^{1,2}_{rad}(\R^N)$ such that up to a subsequence 
\begin{equation}\label{weakD}\widehat{\phi\apice{m}_{p_j}}^{\kappa}\rightharpoonup \widehat{\phi} \qquad \mbox{ in  } D^{1,2}_{rad}(\R^N) \quad\mbox{ as $j\rightarrow +\infty$}
\end{equation}
and so, by the continuous embedding $D^{1,2}_{rad}(\R^N)\hookrightarrow L^2_{\frac{1}{|x|}}(\R^N)$ (defined in \eqref{def:L2Pesato}), we also have 
\begin{eqnarray}
 &&\widehat{\phi\apice{m}_{p_j}}^{\kappa}\rightharpoonup \widehat{\phi} \qquad \mbox{ in  } L^2_{\frac{1}{|x|}}(\R^N) \quad\mbox{ as $j\rightarrow +\infty$}.\label{weakPeso}
 \end{eqnarray}
Moreover, for any bounded set $M\subset \R^N$, by the compact embedding $H^1(M)\hookrightarrow L^2(M)$ we have
\begin{equation}\label{M}\widehat{\phi\apice{m}_{p_j}}^{\kappa}\rightarrow \widehat{\phi} \qquad \mbox{ in  } L^{2}(M)\quad\mbox{ as $j\rightarrow +\infty$}
\end{equation}
and so also
\begin{equation}\label{limiteae} \widehat{\phi\apice{m}_{p_j}}^{\kappa}\rightarrow \widehat{\phi}\qquad a.e. \mbox{ in }\R^N \quad\mbox{ as $j\rightarrow +\infty$}.
\end{equation}
Observe that by \eqref{limiteae} $\widehat{\phi}\geq 0$.
Next we show that
\begin{equation}\label{finonnulla}\widehat{\phi}\not\equiv 0.
\end{equation}
Indeed by assumption \eqref{case1} 
\begin{equation}\label{BellaCondizioneN=2}
\int_{\{|x|\in[\frac{1}{K_{\varepsilon}},K_{\varepsilon}]\}}\frac{\widehat{\phi\apice{m}_{p_j}}^{\kappa}(x)^2}{|x|^2} dx
\geq \alpha_{\varepsilon}>0,\quad \mbox{ for any }j\in\mathbb N.
\end{equation}
Hence taking $M=\{x\in\R^N\, :\, |x|\in[\frac{1}{K_{\varepsilon}},K_{\varepsilon}]\}$, by \eqref{M} we have, as $j\rightarrow +\infty$, that
\[\int_{\{|x|\in[\frac{1}{K_{\varepsilon}},K_{\varepsilon}]\}}\!\!\!\!\frac{\widehat{\phi\apice{m}_{p_j}}^{\kappa}(x)^2}{|x|^2} dx\leq K_{\varepsilon}^2  \int_{\{|x|\in[\frac{1}{K_{\varepsilon}},K_{\varepsilon}]\}}\!\!\!\!\widehat{\phi\apice{m}_{p_j}}^{\kappa}(x)^2 dx \longrightarrow K_{\varepsilon}^2  \int_{\{|x|\in[\frac{1}{K_{\varepsilon}},K_{\varepsilon}]\}}\!\!\!\!\widehat{\phi}(x)^2 dx.\]
Combining this with \eqref{BellaCondizioneN=2} we get
\[ \int_{\{|x|\in[\frac{1}{K_{\varepsilon}},K_{\varepsilon}]\}}\widehat{\phi}(x)^2 dx\geq \frac{\alpha_{\varepsilon}}{K_{\varepsilon}^2} >0,\]
thus proving \eqref{finonnulla}.
\\
We pass to the limit as $j\rightarrow +\infty$ into \eqref{equaziocappuccio} as follows. By Lemma \ref{lemma:beta1pbounded} there exists $\widetilde\beta_1\apice{m}\leq 0$ such that up to a subsequence
\begin{equation}\label{limiteBeta1}
\widetilde\beta_1(m,p_j)\to \widetilde\beta_1\apice{m}\qquad\mbox{as $j\to+\infty$,}
\end{equation}
by \eqref{weakD}
\[
\int_{\R^N \setminus\{0\}}\nabla \widehat{\phi\apice{m}_{p_j}}^{\kappa}\,\nabla\rho\ dx
\ \rightarrow
\int_{\R^N \setminus\{0\}}\nabla \widehat{\phi}\,\nabla\rho\ dx \ \ \mbox{ as }\ j\rightarrow +\infty,
\]
by \eqref{weakPeso}
\begin{equation}\label{convdeboleL2peso}
\int_{\R^N \setminus\{0\}}\frac{\widehat{\phi\apice{m}_{p_j}}^{\kappa}\,\rho}{|x|^2}dx\rightarrow \int_{\R^N \setminus\{0\}}\frac{\widehat{\phi}\,\rho}{|x|^2}dx\ \ \mbox{ as }\ j\rightarrow +\infty,
\end{equation}
for any test function $\rho$ as in \eqref{equaziocappuccio}.
Finally we show that
\[
\int_{\R^N \setminus\{0\}}V\apice{m}_{\kappa,p_j}(x)\,\widehat{\phi\apice{m}_{p_j}}^{\kappa}\,\rho\ dx\ \rightarrow  \int_{\R^N \setminus\{0\}}V(x)\,\widehat{\phi}\,\rho\ dx\ \ \mbox{ as }\ j\rightarrow +\infty,
\]
 where $V(x)$ is the potential defined in \eqref{limiteV}. Indeed:
\begin{eqnarray*}
&&\left|\int_{\R^N \setminus\{0\}}V\apice{m}_{\kappa,p_j}(x)\,\widehat{\phi\apice{m}_{p_j}}^{\kappa}\,\rho\ dx\ - \int_{\R^N\setminus\{0\}}V(x)\,\widehat{\phi}\,\rho\ dx\right|\leq
\\
&&
\qquad \leq \sup_{\rm{\supp } (\rho)} \left( |x|^2|V\apice{m}_{\kappa,p_j}(x)-V(x)| \right) \int_{\R^N \setminus\{0\}} \frac{\widehat{\phi\apice{m}_{p_j}}^{\kappa}\,|\rho|}{|x|^2}\ dx\
+ \\ 
&&\qquad\qquad\qquad\qquad\qquad\qquad\qquad\qquad +\  \left|    \int_{\R^N \setminus\{0\}} \frac{    (\widehat{\phi\apice{m}_{p_j}}^{\kappa} - \widehat\phi)\overbrace{|x|^2V(x) \rho(x)}^{:=\widetilde\rho(x)}}{|x|^2} dx\right|\\
&&
\qquad \leq \sup_{\rm{\supp } (\rho)}\left( |x|^2|V\apice{m}_{\kappa, p_j}(x)-V(x)|\right) \ C_{\rho}\ \left\| \frac{\widehat{\phi\apice{m}_{p_j}}^{\kappa}}{|x|}\right\|_{L^2(\R^N)} 
\!\!\! + \  \left|    \int_{\R^N \setminus\{0\}} \!\!\!\!\frac{    (\widehat{\phi\apice{m}_{p_j}}^{\kappa} - \widehat\phi)\widetilde\rho}{|x|^2} dx\right|\\
&&\qquad\longrightarrow 0 \ \ \ \ \ \mbox{ as $j\rightarrow +\infty$,}
\end{eqnarray*}
where for the first term we have used \eqref{normaPesata1} and the convergence result in \eqref{W_ptoWN} (observe that $supp(\rho)\subset (\widetilde{T}\apice{m}_{\kappa,p_j}\cap\widehat{A\apice{m}_{p_j}}^{\kappa})$ and so  
$V\apice{m}_{\kappa, p_j}$ satisfies \eqref{defVp}) while for the second term the convergence follows from \eqref{convdeboleL2peso} since $\widetilde\rho:=\rho |x|^2V(x) \in C^{\infty}_0(\R^N\setminus\{0\})$.\\
As a consequence by passing to the limit into \eqref{equaziocappuccio} we get
\begin{equation}
\int_{\R^N \setminus\{0\}} \nabla \widehat{\phi}\, \nabla\rho\ dx
-
\int_{\R^N \setminus\{0\}} V(x)\,\widehat{\phi}\, \rho\ dx
-
\widetilde\beta_1\apice{m} \int_{\R^N \setminus\{0\}}
\frac{\widehat\phi\,\rho}{|x|^2}\ dx  =0, 
\end{equation}
for any $\rho\in C^{\infty}_0 (\R^N\setminus\{0\})$, 
namely $\widehat\phi$ is (a weak and so classical) nontrivial nonnegative solution to the limit equation
\begin{equation}\label{equazioneLimitediphiCappuccio}
 -\widehat{\phi}''-\frac{N-1}{s}\widehat{\phi}'-V(s)\widehat{\phi}=\widetilde\beta_1\apice{m}\frac{\widehat{\phi}}{s^2}\qquad s\in (0, +\infty).
 \end{equation}
where $\widetilde\beta_1\apice{m}$ satisfies \eqref{limiteBeta1}.\\
By Theorem \ref{lemma:betastar} (see \eqref{unicheSoluzioniDiEquazioneLimite}) it follows that  $\widetilde \beta_1\apice{m}= -(N-1)$ namely, up to a subsequence
\[\widetilde\beta_1(m,p_{j})\rightarrow -(N-1)\ \ \mbox{ as }\ j\rightarrow +\infty,\]
 thus obtaining \eqref{conclusion:case1}.

\

\

{\emph{Proof in CASE 2.}}

\

Let $
\widetilde\beta^*$ be as in \eqref{defbetastar}, then by Theorem \ref{lemma:betastar}  we know that $\widetilde\beta^*=-(N-1)$ and so, taking $\widehat{\phi\apice{m}_{p_j}}^{0}$ as in \eqref{fiCappuccio}, we have
\begin{eqnarray}\label{stimadalbassoparziale}
-(N-1)&\stackrel{\mbox{\scriptsize{Theorem \ref{lemma:betastar}}}}{=}&\widetilde\beta^* \stackrel{\eqref{defbetastar}+\eqref{normaPesata1}}{\leq} \int_{\mathbb R^N}\left(|\nabla\widehat{\phi\apice{m}_{p_j}}^{0}(x)|^2-V(x)\widehat{\phi\apice{m}_{p_j}}^{0}(x)^2\right) dx
\nonumber
\\
&\stackrel{\eqref{eq cappuccio p}}{=}&\widetilde{\beta}_1(m,p_j)+ \int_{\widehat{A\apice{m}_{p_j}}^{0}}\left[V\apice{m}_{0,p_j}(x)-V(x)\right]\widehat{\phi\apice{m}_{p_j}}^{0}(x)^2 dx,
\end{eqnarray}
where the set $\widehat{A\apice{m}_{p_j}}^{0}$ is defined in \eqref{defApCapp}, $V\apice{m}_{0,p_j}$  satisfies \eqref{defVp} in $\widehat{A\apice{m}_{p_j}}^{0}$ and $V$ is as in \eqref{limiteV}.

\

Next we estimate the term  $\int_{\widehat{A\apice{m}_{p_j}}^{0}}\left[V\apice{m}_{0,p_j}(x)-V(x)\right]\widehat{\phi\apice{m}_{p_j}}^{0}(x)^2 dx$. As before $\varepsilon>0$ is fixed as in \eqref{stimasx}. Let  $R_{\varepsilon}$ be as in Lemma \ref{proposition:partePositiva} and fix $R>0$ such that 
\begin{equation}\label{ourchoiceofR}
R \geq\max\{1, R_{\varepsilon},N(N-2), 
\frac{N\sqrt{(N+2)(N-2)}}{\sqrt{\varepsilon}}\}.
\end{equation} We have
\begin{eqnarray*}
\int_{\widehat{A\apice{m}_{p_j}}^{0}}\left[V\apice{m}_{0,p_j}(x)-V(x)\right]\widehat{\phi\apice{m}_{p_j}}^{0}(x)^2 dx
&
\leq & \int_{\widehat{A\apice{m}_{p_j}}^{0}\cap\{|x|\leq R\}} \left|V\apice{m}_{0,p_j}(x)-V(x)\right|\widehat{\phi\apice{m}_{p_j}}^{0}(x)^2 dx
\\ 
&& \ +\
\int_{\widehat{A\apice{m}_{p_j}}^{0}\cap\{|x|> R\}} V(x)\widehat{\phi\apice{m}_{p_j}}^{0}(x)^2 dx
\\
&&\ +\
\int_{\widehat{F\apice{m}_{p_j}}^{0}\cap\{|x|> R\}} V\apice{m}_{0,p_j}(x) \widehat{\phi\apice{m}_{p_j}}^{0}(x)^2 dx
\\
&&\  +\
\int_{\widehat{T\apice{m}_{p_j}}^{0} \cap\{|x|> R\}} V\apice{m}_{0,p_j}(x) \widehat{\phi\apice{m}_{p_j}}^{0}(x)^2 dx
\\
&& \\
&=& I_j\ +\ II_j \ + \ III_j\ +\ IV_j,
\end{eqnarray*}
where the set $\widehat{F\apice{m}_{p_j}}^{0}$  is as in \eqref{def:Pprisclata}
 while the set  $\widehat{T\apice{m}_{p_j}}^{0} $ is the scaling of the remaining set $A\apice{m}_{p_j}\setminus B\apice{m}_{0,p_j} $ with respect to the same scaling parameter $M\apice{m}_{0,p_j}$. Namely
\[
\widehat{T\apice{m}_{p_j}}^{0}  :=(M\apice{m}_{0,p_j})^{\frac{p_j-1}{2}}\big( A\apice{m}_{p_j}\setminus B\apice{m}_{0,p_j}\big).
\] 
Then
\begin{align*}
I_j =\ & \int_{\widehat{A\apice{m}_{p_j}}^{0}\cap\{|x|\leq R\}} \left|V\apice{m}_{0,p_j}(x)-V(x)\right||x|^2\frac{\widehat{\phi\apice{m}_{p_j}}^{0}(x)^2 }{|x|^2}dx\\
\leq\ & \sup_{B_R(0)}\left|V\apice{m}_{0,p_j}(x)-V(x)  \right|R^2\int_{\mathbb R^N}
\frac{\widehat{\phi\apice{m}_{p_j}}^{0}(x)^2}{|x|^2} dx\\
\overset{\eqref{normaPesata1}}{=} &  \sup_{B_R(0)}\left| V\apice{m}_{0,p_j}(x)-V(x)  \right|R^2
\stackrel{\eqref{V_ptoVN}}{\leq }\varepsilon
\end{align*}
for $j$ sufficiently large.\\
Observe that the radial function $|x|\mapsto V(x)|x|^2\rightarrow 0$  has  a unique maximum for $|x|=N(N-2)$, hence by our choice of $R$ in \eqref{ourchoiceofR}
\[\sup_{\{|x|>R\}}( V(x)|x|^2)\overset{\eqref{ourchoiceofR}}{\leq} V(R)R^2\leq \frac{N^2(N+2)(N-2)}{R^2}\overset{\eqref{ourchoiceofR}}{\leq}\varepsilon\]
and so, for any $j\in\mathbb N$:
\begin{align*}
II_j\ =& \int_{\widehat{A\apice{m}_{p_j}}^{0}\cap\{|x|> R\}} V(x)|x|^2\frac{\widehat{\phi\apice{m}_{p_j}}^{0}(x)^2}{|x|^2} dx\\
 \leq\ &
\sup_{\{|x|>R\}}( V(x)|x|^2) \int_{\widehat{A\apice{m}_{p_j}}^{0}\cap\{|x|> R\}} \frac{\widehat{\phi\apice{m}_{p_j}}^{0}(x)^2}{|x|^2} dx
 \\
 \leq\ & \varepsilon \int_{\R^N}\frac{\widehat{\phi\apice{m}_{p_j}}^{0}(x)^2}{|x|^2} dx\\
 \overset{\eqref{normaPesata1}}{=} &\   \varepsilon.
\end{align*}
By our choice of $R$ in \eqref{ourchoiceofR} we may also apply Lemma \ref{proposition:partePositiva} getting, for $j$ large enough:
\begin{eqnarray*}
III_j =\int_{\widehat{F\apice{m}_{p_j}}^{0}\cap\{|x|> R\}} V\apice{m}_{0,p_j}(x) \widehat{\phi\apice{m}_{p_j}}^{0}(x)^2 dx
\leq \varepsilon 
\end{eqnarray*}
In order to estimate the term $IV_j$  we need all the results about the function $f\apice{m}_{p_j}$ defined in \eqref{definition:f_p}. To this purpose let us observe that the number $K_{\varepsilon}$  in  \eqref{assumptionCASE2} has been chosen so that both Lemma 
\ref{proposition:StimaMaxASxDiEpsilon_p^-} and Lemma \ref{proposition:StimaMaxADxDiEpsilon_p^-} hold. Moreover since
\begin{equation}\label{altroModo} 
A\apice{m}_{p_j}\setminus B\apice{m}_{0,p_j}=\cup_{i=1}^{m-1} B\apice{m}_{i,p_j} =\left\{y\in\R^N :\ r\apice{m}_{1,p_j} < |y| < 1   \right\},
\end{equation} 
it follows that
\[\widehat{T\apice{m}_{p_j}}^{0}= (M\apice{m}_{0,p_j})^{\frac{p_j-1}{2}}\big(\cup_{i=1}^{m-1} B\apice{m}_{i,p_j}   \big)=\left\{x\in\R^N :\ r\apice{m}_{1,p_j} (M\apice{m}_{0,p_j})^{\frac{p_j-1}{2}} < |x| < (M\apice{m}_{0,p_j})^{\frac{p_j-1}{2}}   \right\}
\]
where by the property $(\mathcal A\apice{m}_1)$ (which holds true by Proposition \ref{proposition:AmiValePerOgnim}) one has
 \begin{equation} \label{primaos}
r\apice{m}_{1,p_j} (M\apice{m}_{0,p_j})^{\frac{p_j-1}{2}}>R, \quad\mbox{for $j$ sufficiently large}.
 \end{equation}
As a consequence, for $j$ sufficiently large, we have:
\begin{eqnarray}\label{laprimadi4j}
IV_j &=& \int_{\widehat{T\apice{m}_{p_j}}^{0}  \cap\{|x|> R\}} V\apice{m}_{0,p_j}(x)\widehat{\phi\apice{m}_{p_j}}^{0}(x)^2 dx
\nonumber
\\
&\overset{\eqref{primaos}}{=} & \int_{\widehat{T\apice{m}_{p_j}}^{0}  } V\apice{m}_{0,p_j}(x)\widehat{\phi\apice{m}_{p_j}}^{0}(x)^2 dx  
\nonumber
\\
&\stackrel{\eqref{definition:f_p}}{=}  & p_j \int_{A\apice{m}_{p_j}\setminus B\apice{m}_{0,p_j} }f\apice{m}_{p_j}(y)\frac{\phi_{{p_j}}(y)^2}{|y|^2} dy
\nonumber
\\
&\stackrel{\eqref{altroModo} }{=}  & p_j \sum_{i=1}^{m-1}   \int_{ B\apice{m}_{i,p_j} }f\apice{m}_{p_j}(y)\frac{\phi_{{p_j}}(y)^2}{|y|^2} dy.
\end{eqnarray}
Let $K_{\varepsilon}$ be as in Lemma \ref{proposition:StimaMaxADxDiEpsilon_p^-} and let us define the sets $G\apice{m}_{i,p_j,\varepsilon}:=G\apice{m}_{i,p_j, K}$  with $K=K_{\varepsilon}$, $i=1\ldots, m-1$, where $G\apice{m}_{i,p_j, K}$ is as in \eqref{defG}. Let us also consider  the
set $H\apice{m}_{i,p_j,\varepsilon}$, $i=1\ldots, m-1$, introduced in \eqref{defH},
by Lemma  \ref{proposition:StimaMaxASxDiEpsilon_p^-}
 and \ref{proposition:StimaMaxADxDiEpsilon_p^-} 
\[\emptyset\neq (G\apice{m}_{i,p_j,\varepsilon}\cup H\apice{m}_{i,p_j,\varepsilon})\subset B\apice{m}_{i,p_j}.\]
From \eqref{laprimadi4j}, for $j$ sufficiently large,  it then follows
\begin{eqnarray*}
IV_j &= &   p_j \sum_{i=1}^{m-1}\int_{G\apice{m}_{i,p_j, K_{\varepsilon}}\cup  H\apice{m}_{i,p_j,K_{\varepsilon}}}
f\apice{m}_{p_j}(y)\frac{\phi_{{p_j}}(y)^2}{|y|^2} dy\\
&& 
+\ 
p_j\sum_{i=1}^{m-1}\int_{\big\{\frac{1}{K_{\varepsilon}}(M\apice{m}_{i,p_j})^{-\frac{p_j-1}{2}}\leq |y|\leq K_{\varepsilon} (M\apice{m}_{i,p_j})^{-\frac{p_j-1}{2}}\big\}}  
f\apice{m}_{p_j}(y)\frac{\phi_{{p_j}}(y)^2}{|y|^2} dy
\\
&\stackrel{\eqref{normalizzazione}+\eqref{Q3Nge3}}{\leq} &
p_S\max_{y\in\bigcup_{i=1}^{m-1} (\overline{G\apice{m}_{i,p_j,K_{\varepsilon}}}\cup \overline{H\apice{m}_{i,p_j,K_{\varepsilon}}})}
 f\apice{m}_{p_j}(y)
 \\
&&+\
p_S\, C \sum_{i=1}^{m-1}\int_{\big\{\frac{1}{K_{\varepsilon}}(M\apice{m}_{i,p_j})^{-\frac{p_j-1}{2}}\leq |y|\leq K_{\varepsilon} (M\apice{m}_{i,p_j})^{-\frac{p_j-1}{2}}\big\}}
 \frac{\phi_{{p_j}}(y)^2}{|y|^2} dy
\\
&\overset{(*)}{\leq} & 5\;\varepsilon
\ +\ 5\, C\sum_{i=1}^{m-1}\int_{\{ |x|\in [\frac{1}{K_{\varepsilon}}, K_{\varepsilon} ]\}}\frac{\widehat{\phi\apice{m}_{p_j}}^{i}(x)^2}{|x|^2} dx,
\end{eqnarray*}
where in $(*)$ we have used that $p_S\leq 5$ for any $N\geq 3$, we have estimated the first term by Lemma \ref{proposition:StimaMaxASxDiEpsilon_p^-}
 and \ref{proposition:StimaMaxADxDiEpsilon_p^-} and we have rescaled the second term.  By collecting the estimates in $I_j, II_j, III_j$ and $IV_j$ we then have, for $j$ sufficiently large:
 \begin{eqnarray*}
\int_{\widehat{A_{p_j}^{+}}}\left[V\apice{m}_{0,p_j}(x)-V(x)\right]\widehat{\phi\apice{m}_{p_j}}^{0}(x)^2 dx &\leq &  8\; \varepsilon  + 5\, C\sum_{i=1}^{m-1}\int_{\{ |x|\in [\frac{1}{K_{\varepsilon}}, K_{\varepsilon} ]\}}\frac{\widehat{\phi\apice{m}_{p_j}}^{i}(x)^2}{|x|^2} dx\\
&\leq & 9\;\varepsilon.
\end{eqnarray*}
where the last inequality follows by the assumption \eqref{assumptionCASE2}.
Combining  this result with \eqref{stimadalbassoparziale} we have then proved that there exists $j_{\varepsilon}\in\mathbb N$ such that:
\[
\widetilde{\beta}_1(m,p_j) \geq -(N-1) -9\;\varepsilon\; ,\qquad\mbox{ for }\ j\geq j_{\varepsilon},
\]
namely we have obtained \eqref{conclusion:case2}.
\end{proof}

\

\begin{remark}\label{rmkDifferenceN2}
We stress that Proposition \ref{theorem:limiteBeta1}  does not hold in dimension $N=2$, when $p\rightarrow +\infty$. 
Indeed in the $2$-dimensional case and when $m=2$  it is proved  in \cite[Theorem 6.1]{DeMarchisIanniPacellaMathAnn} that
$\lim_{p\rightarrow +\infty}\widetilde{\beta}_1(2,p) =-\frac{\ell^2 +2}{2}< -1$, for a  number $\ell>0$ which is explicitly computed. 
%
\end{remark}

\

\

\section{Proof of Theorem \ref{teoPrincipaleMorse}}\label{section:ProofMain}
\label{subsectionproofM}

\

\begin{proof} Let $u\apice{m}_p$ be a solution of \eqref{problem} with $m\in\N^+$ nodal regions and $p\in (1,p_S)$. As explained in Section \ref{Section:Asymptotic4} we approximate the ball $B$ by the annulus 
$A_n$ choosing $n=n\apice{m}_p$, where $n\apice{m}_p$ is defined in \eqref{np}, and we consider the radial weighted linear operators $\widetilde L\apice{m}_{p,rad}$ defined in \eqref{defOpRadBreve}.
The eigenvalues of $\widetilde L\apice{m}_{p,rad}$, as in \eqref{np1}, are
\[\widetilde{\beta_i}(m,p),\ \mbox{ for any }i\in\mathbb N^+.\] 
We also set $\widetilde L\apice{m}_{p}:=\widetilde{{L^n}}\apice{m}_{\!\!\!\! p}$ for $n=n\apice{m}_p$, where $\widetilde{{L^n}}\apice{m}_{\!\!\!\! p}$ is the weighted operator defined in \eqref{weightedOp}, whose eigenvalues we denote by
\[\widetilde{\mu_i}(m,p):=\widetilde{\mu_i^{n}}(m,p),  \ \mbox{ for }n=n\apice{m}_p,\quad\mbox{for any $i\in\mathbb N^+$}.
\]
The number of negative eigenvalues of $\widetilde L\apice{m}_{p}$ is then
\[\widetilde{k_p}(m):=\widetilde{k_p^n}(m), \ \mbox{ for }n=n\apice{m}_p,\]
where $\widetilde{k_p^n}(m)$ is as in \eqref{kappatilde}.

\

By Proposition \ref{proposition:MorseRadialeConPeso}  to determine the Morse index $\mathsf{m}(u\apice{m}_p)$ is equivalent to counting the number  $\widetilde{k_p}(m)$ of negative eigenvalues $\widetilde{\mu_i}(m,p)$ of the operator $\widetilde L\apice{m}_{p}$. Hence we should  show that
\begin{equation}\label{tesissima}
\widetilde{k_p}(m)=
            m+ N(m-1)   \quad \hbox{ for  $p$ close to $p_S$.}
\end{equation}

By \eqref{decomposizioneAutovalori}  we have that
\begin{equation}\label{decomposizioneAutovaloriCasoMio} \widetilde{\mu}_j(m,p)\ =\ \widetilde{\beta_i}(m,p)\ +\ \lambda_k, \ \ \mbox{ for } i,j=1,2,\ldots,\ \ k=0,1, \ldots
\end{equation}
where $\lambda_k$ are the eigenvalues of the Laplace-Beltrami operator $-\Delta_{S^{N-1}}$ on the unit sphere $S^{N-1}$, $N\geq 3$. As we already mentioned in \eqref{lambdak}
\[\lambda_k=k(k+N-2) \ (\geq 0),\ \ k=0,1, \ldots \]
with multiplicity (see \cite{BerezinShubin})
\begin{equation}
\label{multiplicityqui}
N_k-N_{k-2}
\end{equation}
where $N_h$, $h\in\mathbb Z$, is defined in \eqref{N_h}.

\

By \eqref{corb} we already know that
\begin{equation} \label{corb3}
\widetilde{k_p}(m)\geq  m+ N(m-1)
\end{equation}
and that
\begin{equation}
\label{corb2}\widetilde{\beta_1}(m,p)\leq\ldots\leq \widetilde{\beta}_{m}(m,p)<\ 0\ \leq\widetilde{\beta}_{m+1}(m,p)\leq \ldots.
\end{equation}

\

By \eqref{corb2}, since $\lambda_k\geq 0$, it immediately follows  that
\begin{equation}
\label{immediato}\widetilde{\beta_i}(m,p)\ +\ \lambda_k,\ \geq \ 0\qquad \forall\, i\geq m+1, \  \ \forall\, k\geq 0
\end{equation}
so that  \emph{all the eigenvalues $\widetilde{\beta_i}(m,p)$ with $i\geq m+1$ cannot produce any negative eigenvalue} $\widetilde{\mu_j}(m,p)$ by the formula \eqref{decomposizioneAutovaloriCasoMio}.

\

Next we analyze the contribution given by the last negative eigenvalue $\widetilde{\beta}_m(m,p)$. Observe that $\lambda_1=N-1$ and, by Proposition  \ref{LemmaStimeAutovaloriRadialeConPeso Iparte},  $\ \widetilde{\beta}_m(m,p)> -(N-1)$, hence we get
\begin{equation}\label{noneg}\widetilde{\beta}_m(m,p)\ +\ \lambda_k\ >\ 0,\ \ \  \ \forall \, k\geq 1.
\end{equation}
On the other side, from \eqref{corb2} and observing that $\lambda_0=0$, we have that
\begin{equation}\label{primoautoneg}
\widetilde{\beta}_m(m,p)\ +\lambda_0\ =\ \widetilde{\beta}_{m}(m,p)\ <\ 0.
\end{equation}
Hence, by \eqref{decomposizioneAutovaloriCasoMio}, \eqref{primoautoneg} gives \emph{one negative eigenvalue of $\widetilde L\apice{m}_{p}$, which is radial and simple}, since by \eqref{multiplicityqui} it follows that  $\lambda_0$ has multiplicity one. Furthermore, because of \eqref{noneg}, this eigenvalue is \emph{the only} negative eigenvalue obtained by summing $\widetilde{\beta}_m(m,p)$ with the eigenvalues of $-\Delta_{S^{N-1}}$

\

Then \eqref{tesissima} is  obviously proved in the case $m=1$.

\

In the case $m\geq 2$ we need to study the remaining negative eigenvalues $\widetilde{\beta_i}(m,p)$, $i=1, \ldots, m-1$ and, since there are exactly $m$ radial simple negative eigenvalues of $\widetilde L\apice{m}_{p}$, we have to prove that they
produce exactly $\ N(m-1)\ $ negative \emph{nonradial} eigenvalues $\widetilde{\mu_j}(m,p)$ by the formula \eqref{decomposizioneAutovaloriCasoMio} (counted with their multiplicity).  
%


\

Since by Proposition \ref{theorem:limiteBeta1} and Corollary \ref{cor:limiteBetai} we have
\begin{equation}\label{stimasxaltri}\liminf_{p\rightarrow p_S}\widetilde{\beta}_i(m,p)\geq -(N-1),\qquad\mbox{ for any } \, i=1,\ldots, m-1
\end{equation}
and observing that $\lambda_k\geq 2N>N-1$ for all $k\geq 2$, it follows  that for $p$ sufficiently close to $p_S$
\begin{equation}\label{info3}
\widetilde{\beta}_i(m, p)\  +\ \lambda_k\ >\  0, \quad\mbox{ for any }i=1,\ldots,m-1, \, \mbox{ for all } k\geq 2.
\end{equation}
 
\

By \eqref{info3} and the estimate \eqref{corb3}  we immediately have that for $p$ close to $p_S$
\begin{equation}\label{stimadx}
\widetilde{\beta}_i(m,p)+\lambda_1< 0, \quad\mbox{ for any } i=1,\ldots, m-1.
\end{equation}
Indeed, since there are exactly $m$ radial simple negative eigenvalues of $\widetilde L\apice{m}_{p}$, by \eqref{corb3} there must be at least $N(m-1)$ negative \emph{nonradial} eigenvalues of $\widetilde L\apice{m}_{p}$ (counted with their multiplicity).
By \eqref{info3}, for $p$ close to $p_S$, these nonradial eigenvalues must be obtained by the formula \eqref{decomposizioneAutovaloriCasoMio} for $i=1,\ldots, m-1$ and  $k=1$ (for $k=0$ only radial eigenvalues may be constructed).
Hence, observing that the multiplicity of $\lambda_1$ is $N$ (by \eqref{multiplicityqui}), we deduce that, if 
\eqref{stimadx} does not hold, then \eqref{corb3} cannot be satisfied.

\

In conclusion by \eqref{info3} and \eqref{stimadx},  for $p$ close to $p_S$ there are exactly 
\emph{$ N(m-1) $ negative nonradial eigenvalues of  $\widetilde L\apice{m}_{p}$}, counted with their multiplicity,
given by
\begin{equation}\label{info22}
\widetilde{\beta}_i(m,p)\  +\ \lambda_1\ <\  0,  \qquad  i=1,\ldots, m-1.
\end{equation}

This proves  \eqref{tesissima} and ends the proof of Theorem \ref{teoPrincipaleMorse}.

\end{proof}

\

\begin{remark}
We point out that, combining \eqref{stimasxaltri}  with \eqref{stimadx} and observing that $\lambda_1=-(N-1)$,  we also get 
\begin{equation}\label{convergenzadasinistraaltri} 
\lim_{p\rightarrow p_S}\widetilde{\beta}_i(m,p) = -(N-1), \quad\ \forall\, i=1,\ldots, m-1,
\end{equation}
as anticipated in Remark \ref{remarkDopoProp4.2}.
\end{remark}

\

{\bf Acknowledgments. }
The authors would like to thank prof. T. Weth for useful discussions and for pointing out references  \cite{Bruning1}, \cite{Bruning2} and \cite{Donnelly}.

\

\

\end{document}